\documentclass[reqno]{amsart}
\usepackage{amsmath, amsthm, amsfonts, amssymb} 
\usepackage{dsfont}
\numberwithin{equation}{section}
\usepackage{url}
\usepackage{hyperref}
\usepackage{enumerate}
\usepackage[utf8]{inputenc}
\usepackage[english]{babel}
\usepackage{graphicx}
\usepackage[usenames, dvipsnames]{color}

\newtheorem{Theorem}{Theorem}[section]
\newtheorem{Cor}[Theorem]{Corollary}
\newtheorem{Prop}[Theorem]{Proposition}
\newtheorem{Lem}[Theorem]{Lemma}

\newtheorem{conj}[Theorem]{Conjecture}

\theoremstyle{remark}
\newtheorem{Rem}[Theorem]{Remark}
\newtheorem{Example}[Theorem]{Example}

\theoremstyle{definition}
\newtheorem{Defn}[Theorem]{Definition}

\newcommand{\coloneqq}{\mathrel{\mathop:}=}

\renewcommand{\epsilon}{\varepsilon}

\newcommand{\R}{\mathbb{R}}

\newcommand{\N}{\mathbb{N}}

\newcommand{\E}{\mathbb{E}}

\renewcommand{\P}{\mathds{P}}

\newcommand{\vect}[1]{\textbf{#1}}

\begin{document}

\title[Quantiles of Wasserstein distance]{Bounding Quantiles of Wasserstein distance between true and empirical measure}

\author[Cohen, Tegn\'er and Wiesel]{Samuel N. Cohen, Martin N.A. Tegn\'er and Johannes C.W. Wiesel}

\maketitle

\begin{abstract}
Consider the empirical measure, $\hat{\P}_N$, associated to  $N$ i.i.d. samples of a given probability distribution $\P$ on the unit interval. For fixed $\P$ the Wasserstein distance between $\hat{\P}_N$ and $\P$ is a random variable on the sample space $[0,1]^N$. Our main result is that its normalised quantiles are asymptotically maximised when $\P$ is a convex combination between the uniform distribution supported on the two points $\{0,1\}$ and the uniform distribution on the unit interval $[0,1]$. This allows us to obtain explicit asymptotic confidence regions for the underlying measure $\P$.

We also  suggest extensions to higher dimensions with numerical evidence.
\end{abstract}



\section{Notation and main result}\label{sec:main}
Many problems in statistics depend on understanding how close the empirical distribution of a sample is from its generating distribution. In this paper, we show that, for measures on $[0,1]$ with distances measured by the $1$-Wasserstein metric, the quantiles of this (random) distance are maximised by a specific family of sampling distributions. This yields a tight upper bound on the law of the Wasserstein distance between the true and empirical measures.

Let $\mathfrak{P}([0,1])$ denote the space of probability measures on $[0,1]$. For two measures  $\P, \tilde{\P}\in \mathfrak{P}([0,1])$ we denote the $1$-Wasserstein distance by
\begin{align*}
\mathcal{W}(\P,\tilde{\P})\coloneqq\inf_{\pi\in \Pi(\P,\tilde{\P})} \left( \int_{\R} |x-y|\pi(dx,dy)\right),
\end{align*}
where $\Pi(\P,\tilde{\P})$ is the set of couplings $\pi\in \mathfrak{P}([0,1]^2)$  with marginals $\P$ and $\tilde{\P}$.
For independent samples $X_1, \dots, X_N$ distributed according to $\P$ we define the empirical measure $$\hat{\P}_N\coloneqq\frac{1}{N}\sum_{k=1}^N \delta_{X_k},$$ where $\delta_{X_k}$ denotes the Dirac measure at $X_k$. As $\hat{\P}_N$ depends on the realisations $X_1,\dots, X_N$, the observed distance $\mathcal{W}(\hat{\P}_N,\P)$ is a random variable with domain $[0,1]^N$. We write $\P^\otimes= \P\otimes \P\otimes \cdots$ for the product measure on $[0,1]^\N$, the space of all sequences of observations. 

Denote weak convergence of a sequence of measures $(\P_N)_{N\in \N}$ to a measure $\P\in\mathfrak{P}([0,1])$ by $\P_N\Rightarrow\P$, and similarly $Y_N\Rightarrow Y$ for convergence in law (under $\P^\otimes$) of random variables. While the well-known Glivenko--Cantelli theorem gives the weak convergence $\hat{\P}_N\Rightarrow\P$ ($\P^\otimes$-a.s.), the problem of controlling the distance between $\hat{\P}_N$ and $\P$ has a long history in probability theory,  statistics and many related fields, see for example \cite{dudley1969speed, delattre2004quantization, pages2012optimal, weed2017sharp, ajtai1984optimal, talagrand1994transportation, biau2008performance, laloe20101, roberts1997shift, bolley2007quantitative, fournier2015rate, fournier2016rate} and the references therein for recent results in quantisation, particle systems and MCMC to name just a few. Building on the results of Dereich, Scheutzow and Schottstedt \cite{dereich2013constructive}, non-asymptotic concentration inequalities of the form
\begin{align*}
\P^{\otimes}\left( \mathcal{W}(\hat{\P}_N,\P)\ge t\right) \le C\exp(-cNt^{2}) \qquad \forall \P\in \mathfrak{P}[0,1]
\end{align*} 
were recently obtained by Fournier and Guillin \cite{fournier2015rate}, however the constants $c,C$ are not explicit. In statistical applications, the  values of these constants along with their (non-)optimality are of importance, for example, for the calculation of confidence intervals.

In this article we try to shed light on this matter by investigating the distribution of $\mathcal{W}(\hat{\P}_N,\P)$ for large $N\in \N$. Specifically, we aim to determine the probability measure $\P$ that maximises the quantile $$F^{-1}_{\mathcal{W}(\hat{\P}_N,\P)}(\alpha)\coloneqq\inf\big\{ x\in [0,1]\ |\ \P^{\otimes}(\mathcal{W}(\hat{\P}_N,\P)\le x)\ge \alpha\big\}$$ for a given confidence level $\alpha\in [0,1]$, i.e.
\begin{align}\label{eq::argmax}
\underset{\P\in \mathfrak{P}([0,1])}{\text{argmax}} F^{-1}_{\mathcal{W}(\hat{\P}_N,\P)}(\alpha).
\end{align}
Our main result is the following:

\begin{Theorem}\label{Theorem::1}
Let $\alpha\in[0,1]$ and let $\ \mathcal{U}([0,1])$ denote the uniform distribution on $[0,1]$. Take $\epsilon>0$. Then there exists $\lambda=\lambda(\alpha, \epsilon)\in[0,1]$ such that 
\begin{align}\label{eq:main}
\P^{\lambda}\coloneqq\lambda \frac{\delta_0 +\delta_1}{2}+(1-\lambda)\ \mathcal{U}([0,1])
\end{align}
asymptotically maximises the normalised quantile $N^{1/2}F^{-1}_{\mathcal{W}(\hat{\P}_N,\P)}(\alpha)$, i.e. there exists $N_0\in \N$ such that for all $N\ge N_0$  
\begin{equation}\label{eq:quantile1}\sup_{\P\in\mathfrak{P}[0,1]}F^{-1}_{\mathcal{W}(\hat{\P}_N,\P)}(\alpha)-F^{-1}_{\mathcal{W}(\hat{\P}_N,\P^{\lambda})}(\alpha)\le \frac{\epsilon}{N^{1/2}}.\end{equation}
Furthermore $\lim_{\alpha\downarrow 0}\lambda(\alpha, \epsilon)=0$ and $\lim_{\alpha\uparrow 1}\lambda(\alpha, \epsilon)=1$ for all $\epsilon>0$.
\end{Theorem}

\begin{Rem}
We note that the interval $[0,1]$ in the statement above can be replaced by any closed and bounded interval. Furthermore Theorem \ref{Theorem::1} can be extended to the case where we only consider $\P\in \mathfrak{P}([0,1])$ with a density bounded from below. We then restrict to $\lambda\in [0,b]$, where $b\in [0,1]$ is chosen such that $\P^{\lambda}$ fulfils this constraint. 

Our method of proof also gives a version of this result when measures are supported on finitely many (equally spaced) points, in which case, $\mathcal{U}([0,1])$ is replaced by a uniform measure over these points.
\end{Rem}

The remainder of this article is structured as follows: Section \ref{sec:finite} is concerned with the proof of Theorem \ref{Theorem::1} for finitely supported measures. This is achieved via the use of a functional central limit theorem, approximating $N^{1/2} \mathcal{W}(\hat{\P}_N,\P)$ in law by a certain integral over a Brownian bridge. As it turns out, we can explicitly find the measure $\P$ maximising the quantile of this integral. We then extend this result to general measures $\P$ on the unit interval $[0,1]$ in Section \ref{sec:proof} using again an approximation argument and a particularly powerful concentration inequality by Talagrand \cite{talagrand1996new}. We also show some numerical results in Section \ref{sec:numerics} and discuss further applications and possible extensions of Theorem \ref{Theorem::1} in Section \ref{sec:applications}.

\section{Brownian bridge approximation and quantile order for measures with finite support}\label{sec:finite}

Even though it is explicitly known, the (rescaled) distribution of $\mathcal{W}(\hat{\P}_N,\P)$ is hard to control for moderate to large $N \in \N$. We thus resort to asymptotic analysis. The following result has a long history in statistical theory and is crucial for our results:
\begin{Lem}[del Bario, Gin\'e , Matr\'an,{\cite[Theorem 2.1]{del1999central}}]\label{Lemma::gine}
For $\P \in \mathfrak{P}([0,1])$ we have
\begin{align*}
 N^{1/2}\mathcal{W}(\hat{\P}_N,\P)\Rightarrow \int_0^1 |B(F_{\P}(t))|dt 
\end{align*}
for $N \to \infty$, where $(B(q))_{0\le q\le 1}$ is a standard Brownian bridge, $F_{\P}$ the distribution function associated with $\P$ and $\Rightarrow$ indicates convergence in law under the measure $\P^\otimes$.
\end{Lem}

Instead of finding a solution to \eqref{eq::argmax} for general measures $\P\in \mathfrak{P}([0,1])$, we first restrict to finitely supported measures: To this end, let us fix $n\in \N$ and take equidistant points \[x_1=0 < x_2=\frac{1}{n-1} <x_3=\frac{2}{n-1}<\cdots < x_{n}=1.\]
The set of all measures $\P\in \mathfrak{P}([0,1])$ supported on $\{x_1, \dots, x_n\}$ is then characterised by the probabilities $p_i \coloneqq \P(\{x_i\})$, where $p_1,\dots, p_n\in [0,1]$ with $\sum_{i=1}^{n} p_i=1$. For such a measure $\P=\sum_{i=1}^n p_i \delta_{x_i}$, Lemma \ref{Lemma::gine} reads
\begin{align*}
N^{1/2}\mathcal{W}(\hat{\P}_N,\P)\Rightarrow \left(\sum_{i=1}^{n-1} \left| B(q_i) \right|(x_{i+1}-x_i)\right)=\sum_{i=1}^{n-1} \frac{\left| B(q_i) \right|}{n-1}
\end{align*}
where $q_1=p_1$ and $q_i=p_1+\dots+p_i$. For notational simplicity,  we  write
\[\mathfrak{B}_n\coloneqq\sum_{i=1}^{n-1}  \frac{| B(q_i) |}{n-1} 
\]
in this case. Furthermore we set $\bar{q}_i=1-q_i$ and define the covariance matrix

\begin{align*}
\pmb{\Sigma}_n^{p_1,\dots,p_n} \coloneqq\begin{pmatrix}
q_1\bar{q}_1 & q_1 \bar{q}_2 & q_1 \bar{q}_3 & \dots & q_1 \bar{q}_{n-1}\\
q_1\bar{q}_2 & q_2\bar{q}_2 &  q_2\bar{q}_3 & \dots & q_2\bar{q}_{n-1}\\
\vdots & \vdots & \vdots & \dots & \vdots\\
q_1 \bar{q}_{n-1} & q_2 \bar{q}_{n-1} &q_3\bar{q}_{n-1} & \dots & q_{n-1} \bar{q}_{n-1}
\end{pmatrix}.
\end{align*}

It is well-known that the discretised Brownian bridge $(B(q_i))_{1\le i\le n-1}$ has covariance matrix $\pmb{\Sigma}_n^{p_1,\dots,p_n}$ and, as its mean is zero,  $\pmb{\Sigma}_n^{p_1,\dots,p_n}$ completely determines its distribution. By the peculiar structure of the covariance matrix $\pmb{\Sigma}_n^{p_1,\dots,p_n} $ there is in fact a one-to-one correspondence between $(p_1,\dots, p_n)$ and $\pmb{\Sigma}_n^{p_1,\dots,p_n}$ (up to symmetry). It is thus sufficient to investigate the tractable properties of $\pmb{\Sigma}_n^{p_1,\dots,p_n}$ in order to determine the dependence of $F^{-1}_{\mathfrak{B}_n}(\alpha)$ on $(p_1,\dots, p_n)$. 

Note that, although $\mathcal{W}(\hat{\P}_N,\P)$ is bounded by one, the support of its normalised version $N^{1/2}\mathcal{W}(\hat{\P}_N,\P)$ is unbounded for $N \to \infty$; the same is true for the distribution of $\mathfrak{B}_n$. 

The proof of Theorem \ref{Theorem::1} for finitely supported measures is carried out in three steps: We first establish \eqref{eq:main} in Proposition \ref{Prop::general} and then prove the limiting behaviour of $\lambda(\alpha)$ for $\alpha\downarrow 0$ and $\alpha\uparrow1$ in Propositions \ref{Prop:2} and \ref{Prop:1}.

\begin{Prop}\label{Prop::general}
Let $\alpha\in [0,1]$ and let $\mathcal{U}(\{x_1,\dots, x_n\})$ denote the uniform distribution on the equidistant points $\{x_1, \dots, x_n\}$. Then there exists $\lambda(\alpha)\in[0,1]$ such that 
$$\P^{\lambda(\alpha)}\coloneqq\lambda(\alpha) \frac{\delta_0 +\delta_1}{2}+(1-\lambda(\alpha))\ \mathcal{U}(\{x_1,\dots, x_n\}) $$
 maximises the quantile $F^{-1}_{\mathfrak{B}_n}(\alpha)$.
\end{Prop}
\begin{proof}
We note that maximising the quantile  $F^{-1}_{ \mathfrak{B}_n}(\alpha)$ for a given $\alpha\in [0,1]$ is equivalent to minimising the distribution function $F_{ \mathfrak{B}_n}(t)$ for a given $t\in \R_+$. Given a vector $(p_1,\dots, p_n)$ of probabilities, as described above, the density of $\mathfrak{B}_n$ is fully specified: indeed define the ball
\[\Xi(t)\coloneqq \{(y_1, \dots, y_{n-1})\in  \R^{n-1} \ | \ |y_1|+\dots +|y_{n-1}| \le t(n-1) \}\quad\text{for }t\ge 0,\]
then
\begin{align}\label{eq:bb}
F_{\mathfrak{B}_n}(t)&
=\P\big((B(q_1), \dots, B(q_{n-1})) \in \Xi(t)\big)\nonumber\\
&=\int_{\Xi(t)} \frac{1}{\sqrt{2 \pi p_1}} e^{\frac{-y_1^2}{2p_1}}
 \frac{1}{\sqrt{2 \pi p_2}} e^{\frac{-(y_2-y_1)^2}{2p_2}}\cdots  \nonumber\\
 &
 \qquad\qquad \cdot\frac{1}{\sqrt{2 \pi p_{n-1}}} e^{\frac{-(y_{n-1}-y_{n-2})^2}{2p_{n-1}}}\frac{1}{\sqrt{p_n}}e^{\frac{-y_{n-1}^2}{2p_n}} dy_1 \cdots dy_{n-1}.
\end{align}

For fixed $t\ge 0$, this expression is symmetric in $p_1$ and $p_n$, and in the middle terms for $p_2,\dots, p_{n-1}$ respectively. We now argue that it is optimal to choose $p_1=p_n$ and $p_2=\dots=p_{n-1}$.
Indeed, let us fix $p_2, \dots, p_{n-1}$, $t\ge 0$ and define the constant $c\coloneqq p_2+\dots+p_{n-1} \ge 0$ as well as the function $$C(y_1, \dots, y_{n-1})\coloneqq  \frac{1}{\sqrt{2 \pi p_2}} e^{\frac{-(y_2-y_1)^2}{2p_2}}\dots  \frac{1}{\sqrt{2 \pi p_{n-1}}} e^{\frac{-(y_{n-1}-y_{n-2})^2}{2p_{n-1}}},$$ which is independent of $p_1$. Differentiating \eqref{eq:bb} with respect to $p_1$ yields
\begin{align}\label{eq:exp}
&\frac{\partial}{\partial p_1 }\left(\int_{\Xi(t)} \frac{1}{\sqrt{2\pi p_1(1-p_1-c)}} e^{-\left(\frac{y_1^2}{2p_1}+\frac{y_{n-1}^2}{2(1-p_1-c)} \right)}C(y_1,\dots, y_{n-1})dy_1\dots dy_n\right)\nonumber\\
&= \int_{\Xi(t)} \frac{1}{\sqrt{2\pi}}e^{-\left(\frac{y_1^2}{2p_1}+\frac{y_{n-1}^2}{2(1-p_1-c)} \right)}C(y_1,\dots, y_{n-1}) \cdot\nonumber  \\
&\qquad \Bigg(-\frac{1-2p_1-c}{2(p_1(1-p_1-c))^{3/2}} +\frac{-\frac{y_1^2}{2p_1^2}+\frac{y_{n-1}^2}{2(1-p_1-c)^2}}{\sqrt{p_1(1-p_1-c)}}
 \Bigg)dy_1\dots dy_{n-1}.
\end{align}
Plugging in $p_1=(1-c)/2$ we obtain
\begin{align*}
\int_{\Xi(t)} \frac{1}{\sqrt{2\pi}}e^{-\left(\frac{y_1^2}{2(1-c)}+\frac{y_{n-1}^2}{2(1-c)} \right)}C(y_1,\dots, y_{n-1}) \Bigg( \frac{-\frac{y_1^2}{(1-c)^2/2}+\frac{y_{n-1}^2}{(1-c)^2/2}}{(1-c)/2}
 \Bigg)dy=0
\end{align*}
by symmetry. We now show that $p_1\mapsto (\partial/\partial p_1) F_{\mathfrak{B}_n}(t;p_1)$ has a global minimum in $p_1$. For this we first show that 
\begin{align*}
\frac{\partial}{\partial p_1} F_{\mathfrak{B}_n}(t; p_1)>0
\end{align*}
for all $p_1 \in ((1-c)/2, 1-c]$, which follows from elementary calculus: indeed we first note that the functions $p_1 \mapsto -(1-2p_1-c)$, $p_1 \mapsto -y_1^2/(2p_1^2)$ and $p_1 \mapsto y_{n-1}^2/(2(1-p_1-c)^2)$ are strictly increasing for every $y_1, y_{n-1} \in \R\setminus \{0\}$. 
As $p_1 \mapsto 1/(p_1(1-p_1-c))$ is strictly increasing for $p_1 > (1-c)/2$ and $C(y_1, \dots, y_{n-1}) > 0$ we thus conclude by \eqref{eq:exp} that for some strictly positive function $p_1 \mapsto \tilde{C}(p_1)$ \begin{align*}
&\frac{\partial}{\partial p_1} F_{\mathfrak{B}_n}(t; p_1)\\
\ge&\ \tilde{C}(p_1)\int_{\Xi(t)}  \Bigg(-\frac{1-2p_1-c}{2(p_1(1-p_1-c))^{3/2}} +\frac{-\frac{y_1^2}{2p_1^2}+\frac{y_{n-1}^2}{2(1-p_1-c)^2}}{\sqrt{p_1(1-p_1-c)}}
 \Bigg)dy_1\dots dy_{n-1}>0
\end{align*} 
for all $p_1\in ((1-c)/2,1]$. Next considering $p_1 \in [0, (1-c)/2)$, we see that (up to exchanging the roles of $y_1$ and $y_{n-1}$, which does not affect the sign) the function $p_1 \mapsto (\partial/\partial p_1) F_{\mathfrak{B}_n}(t; p_1)$ is antisymmetric about $((1-c)/2,0)$. We conclude that $p_1\mapsto F_{\mathfrak{B}_n}(t)$ has a global minimum at $p_1=(1-c)/2$.

By a similar argument  we obtain $p_2=\dots=p_{n-1}$. Finally $p_1\ge p_2$ follows as, again by symmetry considerations, the derivative 
\begin{align*}
\frac{\partial}{\partial p_1 }&\left(\int_{\Xi(t)} \frac{1}{\sqrt{2\pi p_1(1-p_1-c)}} e^{-\left(\frac{y_1^2}{2p_1}+\frac{(y_{2}-y_1)^2}{2(1-p_1-c)} \right)}C(y_1,\dots, y_{n-1})dy_1\dots dy_n\right)
\end{align*}
evaluated at $p_1=(1-c)/2$ is negative: indeed, this can be seen by the same calculation as in \eqref{eq:exp} replacing $y_{n-1}$ with $(y_2-y_1)$ and noting the asymmetry of the terms involving $y_1$ and $y_2$ and the fact that the terms $y_1 \mapsto \exp(-y^2_1/(2p_1))$ and $y_2-y_1 \mapsto \exp(-(y_2-y^2_1)^2/(2(1-p_1-c)))$ are radially symmetric and strictly decreasing as a function of the radius $y_1^2$ and $(y_2-y_1)^2$ respectively.
\end{proof}

In order to finish the proof of Theorem \ref{Theorem::1} for finitely supported measures, we now investigate the behaviour of $\lambda(\alpha)$ when $\alpha$ approaches one of the two boundary points $\{0,1\}$. This is trivial whenever $n\le 2$. We start with the case $\alpha\downarrow 0$:

\begin{Prop}\label{Prop:2}
Let $n\ge 3$. We have $\lim_{\alpha\downarrow 0}\lambda(\alpha)=0$ in Proposition \ref{Prop::general}.
\end{Prop}

\begin{proof}
By Proposition \ref{Prop::general} it is sufficient to restrict to measures 
\begin{equation}\label{eq:plamdefn}
\P^{\lambda}=\lambda\frac{\delta_0+\delta_1}{2}+(1-\lambda)~\mathcal{U}(\{x_1,\dots, x_n\})
\end{equation}
for $\lambda\in[0,1]$. Note that any such distribution can be identified with a probability vector $$p=\left(p_1,\frac{1-2p_1}{n-2},\dots, \frac{1-2p_1}{n-2},p_1\right).$$ We aim to find $p_1$ which minimises
$\P((B(q_1), \dots, B(q_{n-1})) \in \Xi(t))$ for small $t>0$. Substituting and taking the derivative in $p_1$ in \eqref{eq:bb} we have
\begin{align}\label{eq:conclusion}
&\frac{\partial}{\partial p_1}\Bigg(\int_{\Xi(t)}	\frac{p_1^{-1}}{(2\pi)^{(n-1)/2}} 
\left(\frac{1-2p_1}{n-2}\right)^{-(n-2)/2}\nonumber\\
&\qquad\cdot\exp\bigg(-\frac{1}{2}\bigg(\frac{1}{p_1}(y_1^2+y_{n-1}^2)+\frac{n-2}{1-2p_1}\sum_{i=2}^{n-1}(y_i-y_{i-1})^2  \bigg) \bigg)dx\Bigg)\nonumber\\
&=\frac{1}{(2\pi)^{(n-1)/2}}\frac{1}{p_1\left(\frac{1-2p_1}{n-2} \right)^{(n-2)/2}}\cdot  \nonumber\\
&\qquad \int_{\Xi(t)} \bigg( \frac{np_1-1}{(1-2p_1)p_1} +\frac{y_1^2+y_{n-1}^2}{2p_1^2}-\frac{n-2}{(1-2p_1)^2}\sum_{i=2}^{n-1}(y_i-y_{i-1})^2\bigg)\cdot \nonumber\\
& \qquad\qquad\qquad   \tilde{C}(y_1,\dots,y_{n-1})dy_1\dots dy_{n-1},
\end{align}
where 
\[\begin{split}
   0&\leq \tilde{C}(y_1,\dots, y_n)\\
   &\coloneqq \exp\bigg(-\frac{1}{2}\left( \frac{1}{p_1}(y_1^2+y_{n-1}^2)+\frac{n-2}{1-2p_1}\left((y_2-y_1)^2+\dots+(y_{n-1}-y_{n-2})^2\right)\right)\bigg)
  \end{split}\]
and $p_1\in (1/n,1/2)$. Carefully examining \eqref{eq:conclusion} we conclude that for $t=t(p_1)>0$ small enough the derivative $(\partial/\partial{p_1})F_{\mathfrak{B}_n}(t)$ is positive. The claim follows.
\end{proof}

Next we address the case $\lim_{\alpha\uparrow 1}\lambda(\alpha)$. 

\begin{Lem}\label{Lemma::2}
The entries $(q_i\bar{q}_j)_{1\le i\le j \le n}$ of the covariance matrix $\pmb{\Sigma}_n^{p_1,\dots,p_n}$ are maximised componentwise by the choice $p_1=p_n=1/2$ and $p_i=0$ for $i=2,\dots, n-1$. For this choice $q_i\bar{q}_j=1/4$ for all $1\le i \le j\le n$.
\end{Lem}
\begin{proof}
Evidently $q_i \bar{q}_j \le q_i \bar{q}_i\le 1/4$ for $j \ge i$ and the claim follows.
\end{proof}

Before stating the next lemma let us recall the standard notation $f \sim g$ iff $\lim_{t\to \infty} f(t)/g(t)=1$ and $f \lesssim g$ iff $\limsup_{t\to \infty} f(t)/g(t)\le1$ for functions $f,g:\R\to \R$.

\begin{Lem}\label{Lemma:4}
Consider $\pmb{X}\sim \mathcal{N}(\pmb{\mu}, \pmb{\Sigma})$, where $\pmb{\Sigma} \in \R^{d\times d}$. Let $\{\alpha_j \ | \ j=1,\dots, d\}$ be the eigenvalues of $\pmb{\Sigma}$ such that 
\begin{align*}
\alpha_1\ge \alpha_2 \ge \dots \ge \alpha_d.
\end{align*}
Let $m$ be the multiplicity of $\alpha_1$ and denote by $\Gamma(s)=\int _0^{\infty} t^{s-1}e^{-t}dt$ the Gamma function. Then, as $t\to \infty$,
\begin{align}\label{eq:convergence}
\P\left( \|\pmb{X}-\pmb{\mu}\|_{\ell_2}\ge \sqrt{t\alpha_{1}}\right)\sim
\frac{2^{1-\frac{m}{2}}}{\Gamma\left(\frac{m}{2}\right)}e^{-t/2}t^{m/2-1}\prod_{j=m+1}^d \left(1-\frac{\alpha_j}{\alpha_{1}}\right)^{-1/2}.
\end{align}
Furthermore, for fixed $\epsilon>0$, \eqref{eq:convergence} holds uniformly in $\{(\alpha_1, \dots, \alpha_d) \ | \ 1/\epsilon\ge \alpha_1\ge \alpha_1-\epsilon\ge \alpha_{m+1} \ge \cdots \ge \alpha_d\}$.
\end{Lem}

\begin{proof}
The first claim is stated in \cite[Theorem 1]{husler2002formula} and \cite[Theorem 3.6 \& Remark 3.8]{piterbarg1995laplace}. In particular, \cite[Proof of Theorem 1]{husler2002formula} proceeds as follows: Noting that $$\mathcal{L}\left(\|X-\mu\|_{\ell_2}\right)=\mathcal{L}\left(\sum_{i=1}^d \alpha_i U_i^2\right),$$ where $U_i$ are i.i.d. normal random variables, approximation results for the tails of $\chi^2$-distributions can be manipulated to show the claim. A careful examination of the estimates in the proof show that these approximations are indeed uniform in $\{(\alpha_1, \dots, \alpha_d) \ | \ 1/\epsilon\ge \alpha_1\ge \alpha_1-\epsilon\ge \alpha_{m+1} \ge \cdots \ge \alpha_d\}$.
\end{proof}

\begin{Prop}\label{Prop:1}
Let $n\ge 3$. We have $\lim_{\alpha\uparrow 1}\lambda(\alpha)=1$ in Proposition \ref{Prop::general}.
\end{Prop}
\begin{proof}
As in Lemma \ref{Lemma:4} we denote the eigenvalues of 
$$ \pmb{\Sigma}^{\left(\frac{1-\lambda}{n}+\frac{\lambda}{2},\frac{1-\lambda}{n}, \dots, \frac{1-\lambda}{n},\frac{1-\lambda}{n}+\frac{\lambda}{2} \right)}_n$$
by $\{\alpha_i(\lambda) \ | \ i=1,\dots, n-1\}$, where $$\alpha_1(\lambda)\ge \alpha_2(\lambda)\ge \dots\ge \alpha_{n-1}(\lambda).$$
Lemma \ref{Lemma:4} implies that for $t\to \infty$
\begin{align*}
\P^{\lambda}\left( \Big(\sum_{i=1}^{n-1} | B(q_i) |^2\Big)^{1/2}\ge t\right)\sim
\frac{\sqrt{2}}{\Gamma\left(\frac{1}{2}\right)}e^{-\frac{t^2}{2\alpha_1(\lambda)}}\frac{\sqrt{\alpha_1(\lambda)}}{t}\prod_{j=2}^d \left(1-\frac{\alpha_j(\lambda)}{\alpha_{1}(\lambda)}\right)^{-1/2},
\end{align*}
where $\P^\lambda$ are the measures defined in Proposition \ref{Prop::general},
and
\begin{align*}
\P^1(|B(1/2)|\ge t) \sim \frac{\sqrt{2}}{\Gamma\left(\frac{1}{2}\right)}\ \frac{e^{-2t^{2}}}{2t}.
\end{align*}
By the Perron--Frobenius theorem, $\alpha_1(\lambda)$ has multiplicity one and is less than the maximum over the sum of its rows, in particular by Lemma \ref{Lemma::2} it is strictly less than $(n-1)/4$ for $\lambda<1$. Using this fact, together with continuity of the map $$\lambda \mapsto \pmb{\Sigma}^{\left(\frac{1-\lambda}{n}+\frac{\lambda}{2},\frac{1-\lambda}{n}, \dots, \frac{1-\lambda}{n},\frac{1-\lambda}{n}+\frac{\lambda}{2} \right)}_n$$ and thus also continuity of the eigenvalues $\{\alpha_i(\lambda)\ | \ i=1,\dots, n-1\}$ as a function of $\lambda$, and compactness of $[0,1]$, we conclude that the term $$C(\lambda)\coloneqq \frac{\sqrt{2}}{\Gamma\left(\frac{1}{2}\right)} \prod_{j=2}^d \left(1-\frac{\alpha_j(\lambda)}{\alpha_{1}(\lambda)}\right)^{-1/2}$$ is uniformly bounded in $\lambda\in [0,1]$.

Let us write $F_{\mathfrak{B}_n}(t;\lambda)$ for the corresponding distribution functions of $\mathfrak{B}_n=\sum_i |B(q_i)|/(n-1)$, where $q$ is chosen as in Lemma \ref{Lemma::gine}, and $B$ is a Brownian bridge under $\P^\lambda$. By H\"older's inequality we have $\|x\|_{\ell_1}\le \sqrt{n-1} \|x\|_{\ell_2}$ for $x\in \R^{n-1}$. 
We then compute 
\begin{align}\label{eq:as}
\big(1-F_{\mathfrak{B}_n}(t;\lambda) \big)
&\sim \P^\lambda\left(\sum_{i=1}^{n-1} | B(q_i) |\ge t(n-1) \right)\nonumber\\
&\lesssim \P^\lambda\left(\left(\sum_{i=1}^{n-1} | B(q_i) |^2\right)^{1/2}\ge t \sqrt{n-1} \right)\nonumber\\
&\sim  C(\lambda)\ \exp\left(-\frac{t^2(n-1)}{2\alpha_1(\lambda)}\right)\frac{\sqrt{\alpha_1(\lambda)}}{t\sqrt{n-1}}.
\end{align}
As $\mathfrak{B}_n$ is supported on $\R_+$, writing 
$$\bar{\lambda}(t) \coloneqq \text{arg\ inf}_{\lambda\in [0,1]}\big\{F_{\mathfrak{B}_n}(t;\lambda)\big\}$$
for the minimiser at a point $t \in [0,\infty)$, in order to show that 
\[\lim_{t\to \infty}\bar{\lambda}(t) = \lim_{\alpha \to 1}\lambda(\alpha)=1,\]
it is sufficient to show that every sequence $(t_k)_{k\in\N}$ converging to infinity has a subsequence $(t_{k_l})_{l\in\N}$  such that $\lim_{l\to \infty}\bar{\lambda}(t_{k_l})=1$. But if $\lim_{l\to \infty}\bar{\lambda}(t_{k_{l}})<1$ then there exists $\epsilon>0$ and $l_0\in \N$ such that $\bar{\lambda}(t_{k_l})\le 1-\epsilon$ for all $l\ge l_0$. By compactness of the interval $[0,1-\epsilon]$ there furthermore exists $\delta>0$ such that 
$$\frac{\alpha_1(\bar{\lambda})}{(n-1)/4}\le 1-\delta$$ for all $\bar{\lambda}\in [0,1-\epsilon]$. By \eqref{eq:as} above and Lemma \ref{Lemma:4}, we now conclude that, for $l\to \infty$,
\begin{align*}
\left(1-F_{\mathfrak{B}_n}(t_{k_l}; \bar{\lambda}({t_{k_l}}))\right)
&\lesssim \sup_{0\le \bar{\lambda}\le 1-\epsilon} \left(1-F_{\mathfrak{B}_n}(t_{k_l}; \bar{\lambda})\right)\\
&\sim \sup_{0\le \bar{\lambda}\le 1-\epsilon} C(\bar{\lambda})\ \exp\left(-\frac{t_{k_l}^2(n-1)}{2\alpha_1(\bar{\lambda})}\right)\frac{\sqrt{\alpha_1(\bar{\lambda})}}{t_{k_l}\sqrt{n-1}}\\
&\sim \sup_{0\le \bar{\lambda}\le 1-\epsilon} \frac{C(\bar{\lambda})}{C(1)} \left( C(1)~ \frac{\exp\left(-2t_{k_l}^2\right)}{2t_{k_l}}\right)
\frac{\sqrt{\alpha_1(\bar{\lambda})}}{\sqrt{n-1}}\\
&\qquad\qquad\cdot \exp\left(-t_{k_l}^2\left(\frac{n-1}{2\alpha_1(\bar{\lambda})}-2 \right) \right)\\
&\lesssim  \frac{C(\bar{\lambda})}{C(1)} \left(1-F_{\mathfrak{B}_2}(t_{k_l};1) \right)~2\sqrt{1-\delta}~\exp\left(-t_{k_l}^2\left(\frac{2\delta}{1-\delta} \right) \right),
\end{align*}
where the last line follows from noting that $$\frac{n-1}{2\alpha(\bar{\lambda})}-2\ge \frac{2}{1-\delta}-2=\frac{2\delta}{1-\delta}$$ and $C(\bar{\lambda})$ is bounded in $\bar{\lambda}\in [0,1]$. This leads to a contradiction and so we conclude $\bar{\lambda}(t) \to 1$.
\end{proof}

\begin{Rem}\label{rem:quantileapprox}
By rearranging \eqref{eq:as}, as $F_{\mathfrak{B}_2}(t;\lambda)\equiv F_{\mathfrak{B}_2}(t;1)$ we see that as $t\to \infty$
\begin{align*}
&\big(1-F_{\mathfrak{B}_n}(t;\bar{\lambda}(t))\big)\\
& \lesssim\big(1-F_{\mathfrak{B}_2}(t;1)\big)\left\{\sup_{\bar{\lambda}>\inf_{s\ge t}{\bar{\lambda}(s)}} \frac{C(\bar{\lambda})}{C(1)}
\frac{\sqrt{\alpha_1(\bar{\lambda})}}{\sqrt{n-1}}\exp\left(-t^2\left(\frac{n-1}{2\alpha_1(\bar{\lambda})}-2 \right) \right)\right\},
\end{align*}
in particular there exists a continuous function $\bar C$ with $\lim_{t\to \infty}\bar C(t)=1$ such that
\[\big(1-F_{\mathfrak{B}_n}(t;\bar{\lambda}(t))\big)\leq \bar C(t)\big(1-F_{\mathfrak{B}_2}(t;1)\big).\]
 \end{Rem}

As a consequence of the results above, we can identify the maximal choice of $F_{ \mathfrak{B}_n}(t)$ in terms of second order stochastic dominance. For this we recall the definition of a directionally convex function $f:\R^d\to \R$:
\begin{Defn}
Define the difference operator 
\begin{align*}
\Delta_i^{\epsilon}f(x)\coloneqq f(x+\epsilon e_i)-f(x), 
\end{align*}
where $e_i$ is the $i^{\text{th}}$ unit vector and $\epsilon>0$. A function $f:\R^d\to \R$ is called directionally convex if $\Delta_i^{\epsilon}\Delta_j^{\delta}f(x)\ge 0$ for all $x\in\R^n$, $1\le i,j\le n$ and all $\epsilon,\delta>0$.
\end{Defn}
We note that directional convexity neither implies nor is implied by conventional convexity for $d\ge 2$.

\begin{Rem}
Even for $n=3$, the determinant of $\pmb{\Sigma}_3^{1/2,0,1/2}-\pmb{\Sigma}_3^{1/3,1/3,1/3}$ is negative, so there is no hope of using results connecting positive semi-definiteness of matrices and stochastic order. In particular Bergmann \cite[Proposition 2.5]{bergmann} does not apply.
\end{Rem}

We use the following result:

\begin{Lem}[M\"uller, {\cite[Theorem 12]{muller2001stochastic}}]\label{Lemma::3}
Take vectors $\mathbf{X} \sim \mathcal{N}(\pmb{\mu},\mathbf{\Sigma})$ and $\mathbf{Y} \sim \mathcal{N}(\pmb{\mu}', \mathbf{\Sigma}')$. Then $\E(f(\pmb{X}))\le \E(f(\pmb{Y}))$  for all directionally convex functions $f$ if and only if $\pmb{\mu}=\pmb{\mu}'$ and $\sigma_{i,j} \le \sigma_{i,j}'$ for all $i,j \in \{1, \dots, n\}$.
\end{Lem}
\begin{Cor}
For every $(p_1,\dots,p_n)$ and all $K\ge 0$ we have
\begin{equation}\label{eq:seconddom}
\int_K^{\infty} \P^{(p_1,\dots,p_n)}\big( \mathfrak{B}_n \ge t\big)dt\le\int_K^{\infty} \P^{(1/2,0,\dots,0,,1/2)}\big( \mathfrak{B}_n \ge t\big)dt,
\end{equation}
where $\mathfrak{B}_n = \sum_{i=1}^{n-1} \frac{\left| B(q_i) \right|}{n-1} $ and $\{B(q_i)\}_{1\le i\le n-1}$ has covariance matrix $\pmb{\Sigma}^{p_1,\dots,p_n}_n$ under $\P^{(p_1,\dots,p_n)}$.
\end{Cor}

\begin{proof}
From Lemma \ref{Lemma::2} and \ref{Lemma::3} we conclude that the distribution of $(B(q_i))_{1\le i \le n-1}$ is dominated by $\P^{(1/2,0,\dots,0,1/2)}$ in the sense that $\E[f(\pmb{X})]\le \E[f(\pmb{Y})]$ for all directionally convex functions $f$, where $\mathbf{X} \sim \mathcal{N}(\pmb{0}, \mathbf{\Sigma}_n^{p_1,\dots,p_n})$ and $\mathbf{Y} \sim \mathcal{N}(\pmb{0},\pmb{\Sigma}_n ^{1/2,0, \dots,0, 1/2})$. Furthermore we note that, for all $\lambda\in [0,1]$ and all $K\ge 0$,
\begin{align*}
\E_{\P^{(p_1,\dots,p_n)}}\big[( \mathfrak{B}_n-K)^+\big] =\int_K^{\infty} \P^{(p_1,\dots,p_n)}\big(\mathfrak{B}_n \ge t\big)dt.
\end{align*}
As $f(x_1, \dots, x_n)=((|x_1|+\dots +|x_n|)-K)^+$ is directionally convex, we conclude from Lemma \ref{Lemma::3} that \eqref{eq:seconddom} holds for all $K\ge 0$.
\end{proof}

\section{Proof of Theorem \ref{Theorem::1}}\label{sec:proof}
Given the results of Section \ref{sec:finite}, the proof of Theorem \ref{Theorem::1} is mainly concerned with a probabilistic extension of the results obtained in Section \ref{sec:finite} to general measures $\P\in \mathfrak{P}([0,1])$. This is obtained from weak compactness of the set $\mathfrak{P}([0,1])$  and the following powerful concentration inequality for empirical measures:

\begin{Lem}[Talagrand, \cite{talagrand1996new}, Theorem 1.4]\label{lem:talagrand}
Let $\mathbf{X}_1,\dots, \mathbf{X}_N$ be $\P^{\otimes}$-independent random variables with values in a measurable space $(S,\mathcal{S})$, let $\mathcal{F}$ be a countable class of measurable functions on $S$ and let 
\begin{align*}
&Z\coloneqq \sup_{f\in \mathcal{F}} \sum_{i=1}^N f(\mathbf{X}_i),\\
U\coloneqq \sup_{f\in \mathcal{F}}&\|f\|_{\infty} \quad \text{and}\quad V\coloneqq \E\left( \sup_{f\in \mathcal{F}} \sum_{i=1}^N f^2(\mathbf{X}_i)\right).
\end{align*}
Then there exists a universal constant $K$ (independent of $\mathbf{X}_i, N, S$ and $\mathcal{F}$) such that
\begin{align*}
\P^{\otimes}\left(|Z-\E(Z)|\ge t \right)\le K \exp\left(-\frac{t}{KU}\log \left( 1+\frac{tU}{V}\right) \right).
\end{align*}
\end{Lem}

We now prove Theorem \ref{Theorem::1} for general measures $\P\in \mathfrak{P}([0,1])$, by extending our result for finitely supported measures.
\begin{proof}[Proof of Theorem \ref{Theorem::1}]
For any given measure $\P\in \mathfrak{P}([0,1])$ consider a sequence $(\P^n)_{n\in\N}$ of measures supported on the equally spaced grid
\[\mathbf{x}_n = \Big\{x_i=\frac{i-1}{n-1}\Big\}_{i=1}^n = \Big\{0, \frac{1}{n-1}, \frac{2}{n-1}, \dots, 1\Big\}\]
such that we have the weak convergence $\P^n\Rightarrow\P$. Consequently, we have the convergence in law (under $\P^\otimes$)
\begin{align*}
(n-1)^{-1}\mathfrak{B}_n &=\sum_{i=1}^{n-1} \left| B(q_i) \right|(x_{i+1}-x_i)=\int_0^1 |B(F_{\P^n}(t))|dt  \\&\Rightarrow \int_0^1 |B(F_{\P}(t))|dt 
\end{align*}
as $n\to \infty$. This suggests that it is sufficient to consider finitely supported measures.

Now fix $\alpha\in [0,1]$. For every $n \in \N$ we denote by $\mathfrak{P}(\mathbf{x}_n)$ the measures supported on $\mathbf{x}_n\subset [0,1]$:
$$\mathfrak{P}(\mathbf{x}_n)\coloneqq \Big\{\P\in \mathfrak{P}([0,1]) \ \Big| \ \P=\sum_{i=1}^n p_i\delta_{x_i}\Big\}.$$ We conclude by Proposition \ref{Prop::general} that the quantile $F^{-1}_{\mathfrak{B}_n}(\alpha)$ is maximised for measures in $\mathfrak{P}(\mathbf{x}_n)$ by $$\P^{\lambda^n}\coloneqq\lambda^n \frac{\delta_0+\delta_1}{2}+(1-\lambda^n)\,\mathcal{U}(\mathbf{x}_n),$$ where $\lambda^n=\lambda^n(\alpha)\in [0,1]$ and $\mathcal{U}(\mathbf{x}_n)$ is the uniform measure over $\mathbf{x}_n$. After taking a subsequence, $\lambda^n$ converges to some $\lambda=\lambda(\alpha)\in [0,1]$ and thus
\begin{align}\label{eq:Plamn}
 \P^{\lambda^n}&=\lambda^n \frac{\delta_0+\delta_1}{2}+(1-\lambda^n)\,\mathcal{U}(\mathbf{x}_n)\nonumber\\
&\Rightarrow \P^{\lambda}\coloneqq\lambda \frac{\delta_0+\delta_1}{2}+(1-\lambda)\,\mathcal{U}([0,1]).
\end{align}
We wish to show that $\P^\lambda$ is the measure maximising the desired quantile in $\mathfrak{P}([0,1])$.

Next we need to establish that a general measure can be approximated sufficiently well by a measure on $\mathbf{x}_n$, in particular, our intermediate goal is to establish
\begin{align}\label{eq:1}
\lim_{N\to \infty} \sup_{\P\in \mathfrak{P}([0,1])}F^{-1}_{N^{1/2}\mathcal{W}(\hat{\P}_N,\P)} (\alpha)=\lim_{n\to \infty} \lim_{N\to \infty}\sup_{\P\in \mathfrak{P}(\mathbf{x}_n)} F^{-1}_{N^{1/2}\mathcal{W}(\hat{\P}_N,\P)} (\alpha),
\end{align}
where $F^{-1}$ is the quantile under the corresponding measure $\P^\otimes$.
To do this, we follow a similar outline to that in del Barrio, Gin\'e and Matr\'an \cite[Proofs of Theorems 4.2 \& 5.1]{del1999central}, which in turn heavily relies on Lemma \ref{lem:talagrand}. Given our observation random variables $X_i$, define the quantised versions
\begin{align*}
X_i^{\lfloor n\rfloor}\coloneqq \frac{\lfloor nX_i\rfloor}{n}.
\end{align*}
We denote by $\P^{\lfloor n\rfloor}$ and $\hat{\P}^{\lfloor n\rfloor}_N$ the corresponding true and empirical distributions of our quantised observations.
In order to obtain \eqref{eq:1}, we need to show weak convergence to zero of 
\[N^{1/2}\mathcal{W}(\hat{\P}_N, \P)-N^{1/2}\mathcal{W}(\hat{\P}_N^{\lfloor n\rfloor}, \P^{\lfloor n\rfloor})\]
for $n\to \infty$, uniformly in $N$. Writing
\begin{align*}
Z_N\coloneqq N\mathcal{W}(\hat{\P}_N, \P), \qquad Z^{\lfloor n\rfloor}_N = N\mathcal{W}(\hat{\P}_N^{\lfloor n\rfloor}, \P^{\lfloor n\rfloor})
\end{align*}
we are interested in the convergence of $N^{-1/2}|Z_N-Z_N^{\lfloor n\rfloor}|$. 
Define the random variables
\begin{align*}
\mathbf{X}_i(t)&\coloneqq h_{X_i}(t)=\big(\mathds{1}_{\{X_i> t\}}-\P(X_i> t)\big)-\big(\mathds{1}_{\{X^{\lfloor n\rfloor}_i> t\}}-\P(X_i^{\lfloor n\rfloor}> t)\big)\\
& = \mathds{1}_{\{X_i^{\lfloor n\rfloor} \le t < X_i\}}-\E\left(\mathds{1}_{\{X_i^{\lfloor n\rfloor} \le t < X_i\}}\right).
\end{align*}
Since $L_1(\R)$ is separable, by an application of the Hahn--Banach theorem there exists a countable subset of the unit ball of $L_{\infty}(\R)$, which we denote by $\mathcal{F}$, such that $$\|h\|_{L_1(\R)}=\sup_{f\in \mathcal{F}}\langle f,h\rangle\quad \text{for all }h\in L_1(\R).$$ By construction, 
\begin{align}\label{eq:supfbound}
\sup_{f\in \mathcal{F}}\sum_{i=1}^N  \langle f, \mathbf{X}_i\rangle &=\left\|\sum_{i=1}^N   \mathbf{X}_i \right\|_{L_1(\R)}\nonumber\\
&=\int_0^1\bigg| \sum_{i=1}^N \Big(\mathds{1}_{\{X_i> t\}}-\P(X_i> t)-\mathds{1}_{\{X^{\lfloor n\rfloor}_i> t\}}+\P(X_i^{\lfloor n\rfloor}> t)\Big)\bigg|dt\nonumber\\
&=N\int_0^1 \Big|F_{\hat{\P}_N}(t)-F_{\P}(t)-F_{\hat{\P}_N^{\lfloor n\rfloor}}(t)+F_{\P^{\lfloor n\rfloor}}(t)\Big|dt\nonumber\\
&\ge N \bigg|\int_0^1 \left|F_{\hat{\P}_N}(t)-F_{\P}(t)\right| dt-\int_0^1 \left|F_{\hat{\P}_N^{\lfloor n\rfloor}}(t)-F_{\P^{\lfloor n\rfloor}}(t)\right| dt\bigg|\nonumber\\
& = |Z_N-Z^{\lfloor n\rfloor}_N| .
\end{align}

For the family of functions $\mathcal{F}$ and random variables $\{\mathbf{X}_i\}_{i=1}^N$, we next estimate the constants $U$ and $V$ appearing in Lemma \ref{lem:talagrand}. Clearly $U\leq 1$ by definition of $\mathcal{F}$ and
\begin{align*}
V&=\E\bigg(\sup_{f\in \mathcal{F}} \sum_{i=1}^N \langle f, \mathbf{X}_i\rangle^2\bigg)\\
&\le\E\bigg(\sup_{f\in \mathcal{F}}\sum_{i=1}^N \bigg( \int_0^1 |f(t)|\left| \mathds{1}_{\{X_i> t\}}-\P(X_i> t)  -\mathds{1}_{\{X_i^{\lfloor n\rfloor}> t\}}+\P(X^{\lfloor n\rfloor}_i> t) \right|dt\bigg)^2\,\bigg) \\
&\leq \E\bigg(\sum_{i=1}^N \bigg( \int_0^1 \left| \mathds{1}_{\{X_i^{\lfloor n\rfloor} \le t < X_i\}}-\E\left(\mathds{1}_{\{X_i^{\lfloor n\rfloor} \le t < X_i\}}\right) \right|dt\bigg)^2\,\bigg) \\
&\leq  N \int_0^1 \mathrm{Var}\left(\mathds{1}_{\{X_1^{\lfloor n\rfloor} \le t < X_1\}}\right)dt\leq N \int_0^1 \E\left(\mathds{1}_{\{X_1^{\lfloor n\rfloor} \le t < X_1\}}\right)dt\\
&\le \frac{ N}{n}. 
\end{align*}
Similarly, 
we  compute
\begin{align*}
 0&\leq\E\left(\sup_{f\in \mathcal{F}}\sum_{i=1}^N  \langle f, \mathbf{X}_i\rangle\right) \\&=N\int_0^1\E\bigg| \left(\frac{1}{N}\sum_{i=1}^N \mathds{1}_{\{X_i^{\lfloor n\rfloor} \le t < X_i\}}- \E\left(\mathds{1}_{\{X_i^{\lfloor n\rfloor} \le t < X_i\}}\right)\right) \bigg|dt\\
 &\leq N\int_0^1\sqrt{\mathrm{Var}\left(\mathds{1}_{\{X_1^{\lfloor n\rfloor} \le t < X_1\}}\right)/N}dt\leq \sqrt{N}\int_0^1\sqrt{\E\left(\mathds{1}_{\{X_i^{\lfloor n\rfloor} \le t < X_i\}}\right)}dt \\
 &\leq \sqrt{N} \sqrt{\E\left(\int_0^1 \mathds{1}_{\{X_1^{\lfloor n\rfloor} \leq t < X_1\}} dt\right)} \leq \sqrt{\frac{N}{n}}.
 \end{align*}
Combining with \eqref{eq:supfbound}, we have
\[\bigg|\sup_{f\in \mathcal{F}}\sum_{i=1}^N  \langle f, \mathbf{X}_i\rangle - \E\left(\sup_{f\in \mathcal{F}}\sum_{i=1}^N  \langle f, \mathbf{X}_i\rangle\right)\bigg|\geq |Z_N-Z^{\lfloor n\rfloor}_N|  - \sqrt{\frac{N}{n}}.\]
We conclude, using Lemma \ref{lem:talagrand},
that there exists a universal constant $K$ such that
\begin{align}\label{eq:talagrandbasic}
\P^{\otimes}\left( \left|Z_N-Z_N^{\lfloor n\rfloor}\right| \ge t +  \sqrt{\frac{N}{n}}\right)
&\le K\exp\left( -\frac{t}{K}\log\left(1+\frac{tn}{N} \right) \right)
\end{align}
for all $n\in \N$ and $t>0$.  Evaluating \eqref{eq:talagrandbasic} at $t = \sqrt{N/\log(n)}$ yields
\begin{align}\label{eq:talagrand}
&\P^\otimes\left(N^{-1/2}\left|Z_N-Z_N^{\lfloor n\rfloor}\right| \ge \frac{1}{\sqrt{\log(n)}} +\frac{1}{\sqrt{n}}\right)\nonumber\\
&\le K\exp\left( \frac{ -\sqrt{N}}{K\sqrt{\log(n)}}\log\bigg(1+\frac{ n}{\sqrt{N\log(n)}} \bigg) \right)\nonumber\\
&\le K\exp\left( \frac{-1 } {K\sqrt{\log(n)}}\log\bigg(1+\frac{ n}{\sqrt{\log(n)}} \bigg) \right),
\end{align}
where the third line follows from the fact that $x\mapsto x\log(1+c/x)$ is increasing for all $c>0$. In particular, the bound in \eqref{eq:talagrand} is independent of $N$ and the distribution $\P\in \mathfrak{P}([0,1])$, so we have the convergence in probability (under $\P^\otimes$) of  $N^{-1/2}|Z_N^{\lfloor n\rfloor}-Z_N|\to 0$  as $n\to \infty$, uniformly in $N\in\N$ and $\P\in \mathfrak{P}([0,1])$.

We now use this convergence  to establish \eqref{eq:1}. Fix $\epsilon>0, \alpha\in [0,1]$ and take $n\in \N$ large enough that,
\begin{align*}
 \max\bigg\{&\frac{1}{\sqrt{\log(n)}} +\frac{1}{\sqrt{n}}, K\exp\bigg(\frac{-1}{K\sqrt{\log(n)}} \log \bigg( 1+\frac{n}{\sqrt{\log(n)}}\bigg)\bigg)\bigg\}\le \epsilon.
\end{align*}
For any $x\in \R$ such that 
\begin{align*}
\P^\otimes\left(N^{-1/2}Z^{\lfloor n\rfloor}_N\le x \right)\ge\alpha+\epsilon,
\end{align*}
by \eqref{eq:talagrand}  we have
\begin{align*}
&\P^\otimes\left(N^{-1/2}Z_N\le x+\epsilon \right)\\
&\ge \P^\otimes\left(N^{-1/2}Z^{\lfloor n\rfloor}_N\le x,  \quad N^{-1/2}|Z_N^{\lfloor n\rfloor}-Z_N|\le \epsilon \right)\\
&\ge \alpha+\epsilon-\epsilon=\alpha.
\end{align*}
It follows that
\begin{align*}
F^{-1}_{N^{1/2}\mathcal{W}(\hat{\P}_N,\P)}(\alpha-\epsilon)-F^{-1}_{N^{1/2}\mathcal{W}(\hat{\P}_N^{\lfloor n\rfloor},\P^{\lfloor n\rfloor})}(\alpha)\le \epsilon.
\end{align*}
In particular, this implies
\begin{align*}
\lim_{N\to \infty} \sup_{\P\in \mathfrak{P}[0,1]} F^{-1}_{N^{1/2}\mathcal{W}(\hat{\P}_N, \P)} (\alpha)\le \epsilon+\lim_{N\to \infty} \sup_{\P\in \mathfrak{P}(\mathbf{x}_n)}F^{-1}_{N^{1/2}\mathcal{W}(\hat{\P}_N,\P)}(\alpha+\epsilon),
\end{align*}
from which \eqref{eq:1} follows.

Now that we have the desired uniform approximation result \eqref{eq:1}, we can combine \eqref{eq:Plamn} with our result for finitely supported measures (Lemma \ref{Lemma::gine}), to conclude 
\begin{align*}
\lim_{N\to \infty}\sup_{\P\in \mathfrak{P}([0,1])} F^{-1}_{N^{1/2}\mathcal{W}(\hat{\P}_N,\P)}(\alpha) &=\lim_{n\to \infty}\lim_{N\to \infty} \sup_{\P\in \mathfrak{P}(\mathbf{x}_n)} F^{-1}_{N^{1/2}\mathcal{W}(\hat{\P}_N,\P)}(\alpha) \\
&=\lim_{n\to \infty}  \sup_{\P\in \mathfrak{P}(\mathbf{x}_n)} F^{-1}_{\sum_{i=1}^{n-1}\frac{ \left| B(q_i) \right|}{(n-1)}} (\alpha) \\
&=\lim_{n\to \infty}  \sup_{\P\in \mathfrak{P}(\mathbf{x}_n)} F^{-1}_{\int_0^1 |B(F_{\P}(t))|\, dt}(\alpha) \\
 &= \lim_{n\to \infty} F^{-1}_{\int_0^1 |B(F_{\P^{\lambda_n}}(t))|\,dt}(\alpha)\\
&=F^{-1}_{\int_0^1 |B(F_{\P^{\lambda}}(t)|\, dt}(\alpha)\\
 &=\lim_{N\to \infty} F^{-1}_{N^{1/2}\mathcal{W}(\hat{\P}_N,\P^\lambda)}(\alpha).
\end{align*}
For fixed $\epsilon>0$ it is thus sufficient to consider $\P^{\lambda^n}$ instead of $\P^{\lambda}$ for large $n\in \N$ in the statement of Theorem \ref{Theorem::1}. As the asymptotic relations $\lim_{\alpha\downarrow 0}\lambda^n(\alpha)=0$ and $\lim_{\alpha\uparrow 1}\lambda^n(\alpha)=1$ follow from Propositions \ref{Prop:2} and \ref{Prop:1}, this concludes the proof. 
\end{proof}

\section{Numerical results for measures with finite support}\label{sec:numerics}

We illustrate the implications of Theorem \ref{Theorem::1} for measures with finite support, i.e. Proposition \ref{Prop::general}, by a numerical experiment. We also begin to consider measures in higher dimension, optimising  the empirical quantiles associated with some measures in $\mathfrak{P}([0,1]^2)$.

\subsection{Illustration of Proposition \ref{Prop::general} } Here we look at the distribution of the discretised Brownian bridge $(B(q_i))_{1\leq i\leq n-1}$ in a ball $\Xi(t)$  as a function of the radius $t$. To this end, we estimate the integral in \eqref{eq:bb} via Monte Carlo
\begin{equation}\label{eq6}
F_{\mathfrak{B}_n}(t) = \P\big( (B(q_1),\dots,B(q_{n-1})) \in \Xi(t) \big) \approx  \frac{1}{M} \sum_{i=1}^M \mathbf{1}_{\{ \vect{B}^{(i)}\in\Xi(t)\}}, \quad t\geq 0,
\end{equation}
 where $\vect{B}^{(i)}$, $i=1,\dots,M$ are i.i.d. draws from $\mathcal{N}(\mathbf{0},\pmb{\Sigma}_n^{p_1,\dots,p_n})$. Figure \ref{f5} shows the result with a Monte Carlo sample of size $M=10^5$ for $n=10$ (left panel) and $n=1000$ (right panel). In both panels we plot the extremal cases from $\P_{\delta}=\frac{1}{2}(\delta_0+\delta_1)$ in black and from $\P_{\mathcal{U}}=\mathcal{U}(\{x_1,\dots,x_n\})$ in blue. We also include the distribution function from convex combinations  $\P^{\lambda}=\lambda\P_{\delta}+(1-\lambda)\P_{\mathcal{U}}$, $\lambda\in(0,1)$ in red and from
  randomly generated measures in green.
\begin{figure}[h!] 
\makebox[\textwidth][c]{ 
		\includegraphics[height=7.5cm,trim=0 0 10 0,clip]{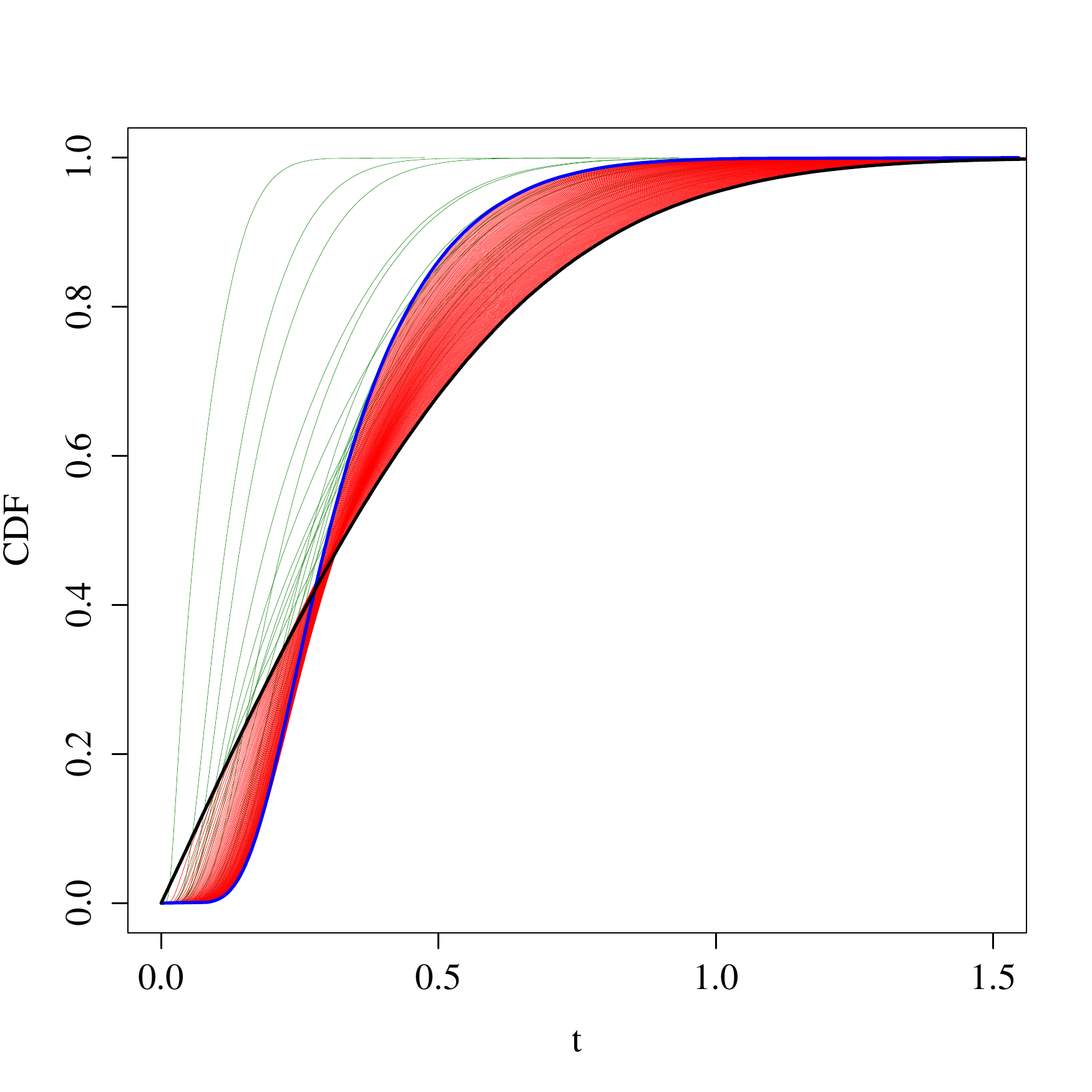}
		\includegraphics[height=7.5cm,trim=20 0 10 0,clip]{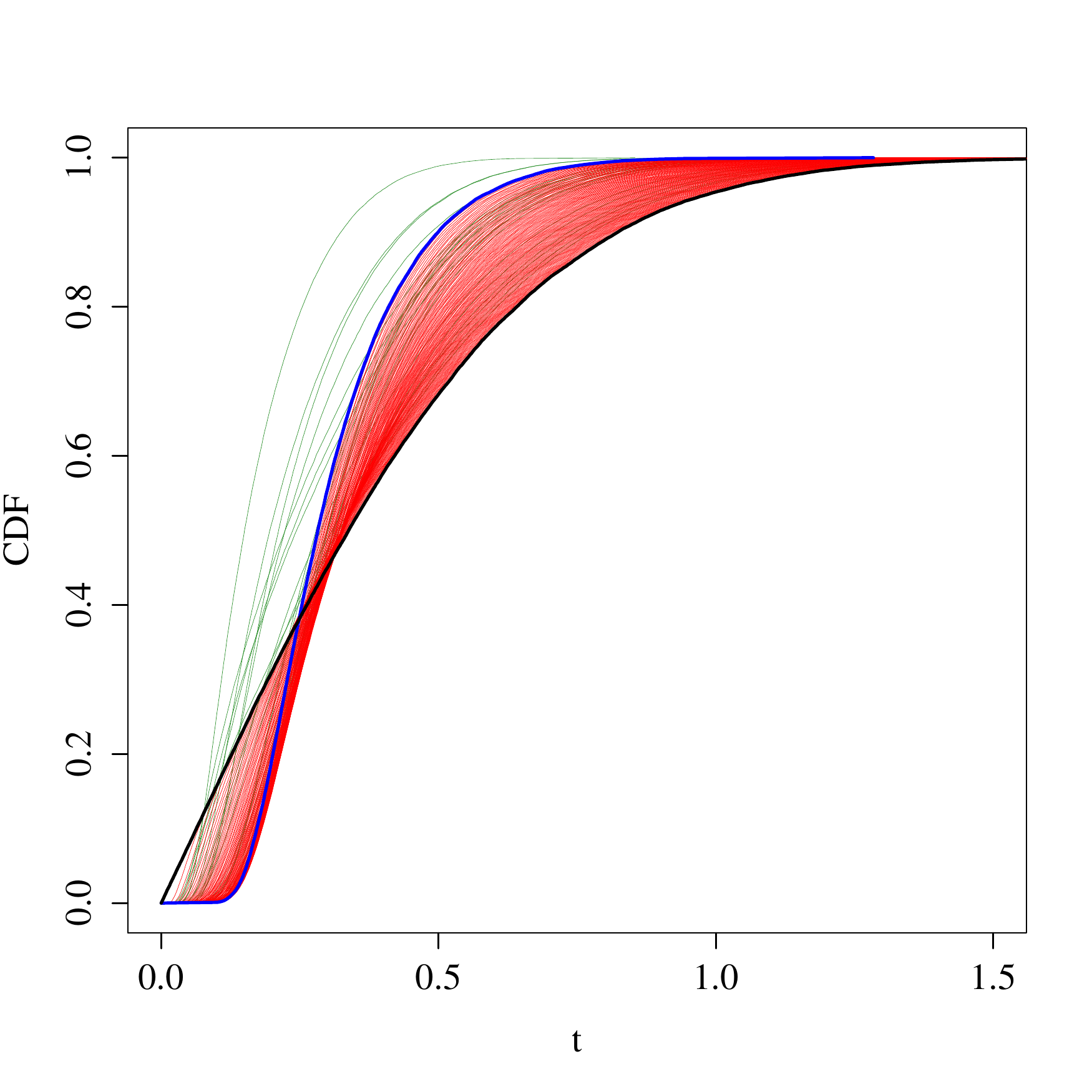}
}
\caption{The distribution function of the rescaled Wasserstein distance $\sum_{i=1}^{n-1} |B(q_i)|(x_{i+1}-x_i)$ for different measures with support $\{x_1,\dots,x_n\}$.  \textbf{Left:} $n=10$,    \textbf{Right:} $n=1000$. In both figures, the case $\P=\frac{1}{2}(\delta_0+\delta_1)$ is  in black and  $\P=\mathcal{U}(\{x_1,\dots,x_n\})$ in blue, and convex combinations of the two in red. Green lines are from 20 different measures generated randomly.}
\label{f5}
\end{figure}
Note that, for $\alpha \approx 0.4$, intermediate distributions with $\lambda\in(0,1)$ yield the most extreme quantiles. This is highlighted in Figure \ref{f6} for the case $n=10$, where we plot the $\lambda$ which  generates the dominating measure $\P^{\lambda}$ as a function of $\alpha\in(0,1)$. (The `step function' behaviour of the graph is due to numerical noise.) Note the behaviour $\lambda(\alpha)\uparrow 1$ as $\alpha\uparrow 1$ and $\lambda(\alpha)\downarrow 0$ as $\alpha\downarrow 0$, that is, dominance by the two extreme cases $\P_{\delta}$  and $\P_{\mathcal{U}}$ respectively. This is as expected from  Theorem \ref{Theorem::1}.
\begin{figure}[h!] 
\makebox[\textwidth][c]{ 
		\includegraphics[width=7.5cm,trim=0 0 0 0,clip]{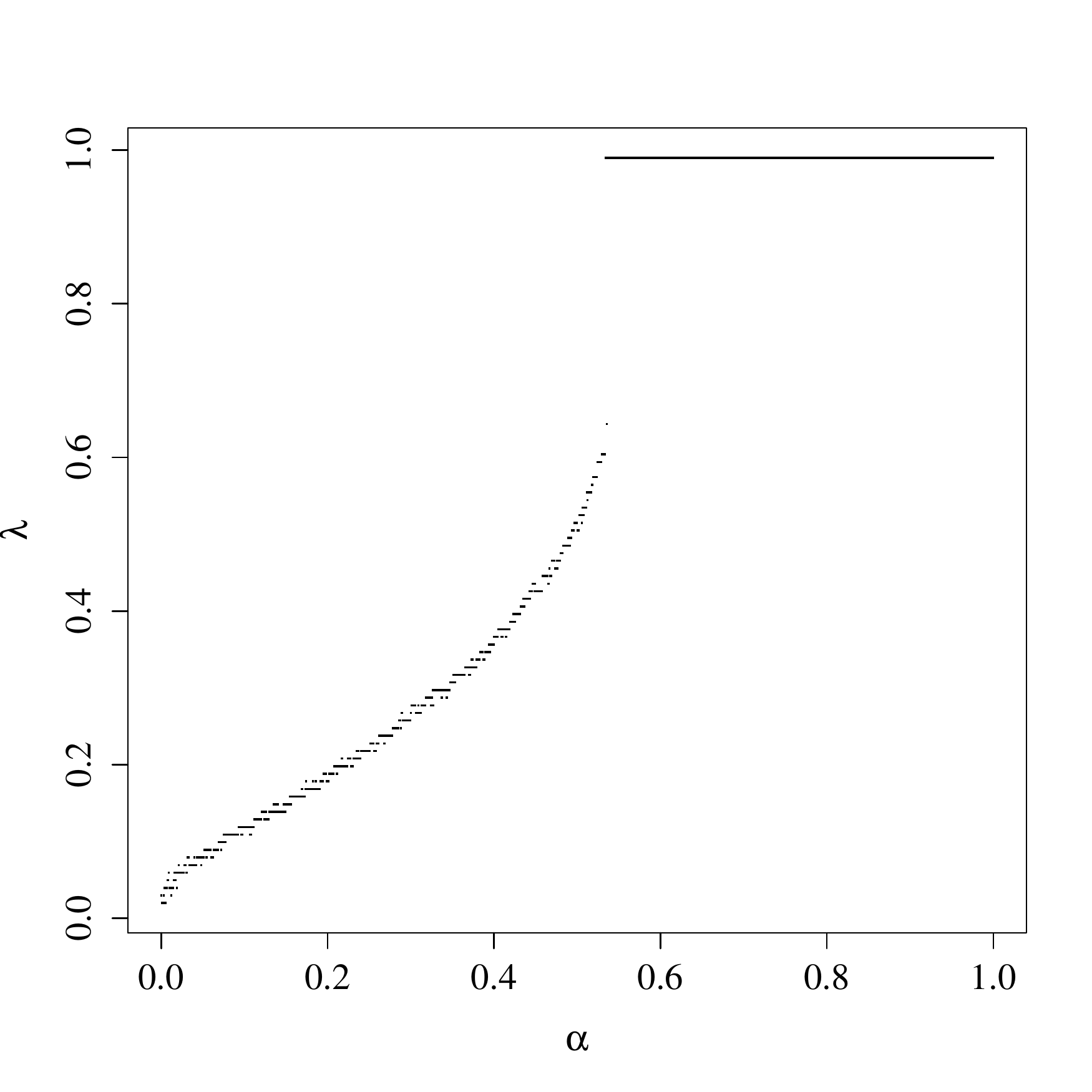}  
}
\caption{Estimated value of $\lambda(\alpha)$  at each confidence level $\alpha$, as in Proposition \ref{Prop::general}, computed in the case $n=10$.}
\label{f6}
\end{figure}

\subsection{Empirical quantiles for measures on the unit square}
Consider measures in  $\mathfrak{P}([0,1]^2)$ with finite support $\vect{x}\subset[0,1]^2$. For simplicity, we take $\vect{x}$ to be a  Cartesian product 
$$\vect{x}  = \{(x_i,y_j):i=1,\dots,n,j=1,\dots,m\},$$
which we index  as a $n\times m$ matrix with $(x,y)$-elements  
$$\vect{x}_{ij} = (x_i,y_j).$$
A measure $\P\in\mathfrak{P}([0,1]^2)$ with support $\vect{x}$ can then be identified with  a $n\times m$ matrix $\vect{p}$  with elements $p_{ij} \in [0,1]$ summing up to one 
$$\P(\{\vect{x}_{ij}\}) = \P(\{(x_i,y_j)\})={p}_{ij}, \quad \sum_{i,j} {p}_{ij} = 1.$$
We denote by $\mathfrak{P}(\vect{x})$ the set of all such probability matrices. 

The Wasserstein distance with $\ell_1$-norm between $\P$ and a measure $\tilde{\P}$ with support $\tilde{\vect{x}}$ of size $\tilde{n}\times\tilde{m}$ and probabilities  $\tilde{\vect{p}}\in\mathfrak{P}(\tilde{\vect{x}})$  is given by
\begin{equation}\label{eqq2}
\mathcal{W}(\P,\tilde{\P}) = \inf_{\boldsymbol{\pi}\in\Pi(\P,\tilde{\P})}  \sum_{i,j,k,l} ||\vect{x}_{ij}-\tilde{\vect{x}}_{kl}||_{\ell_1} \, \pi_{ijkl} 
\end{equation}
where $\Pi(\P,\tilde{\P})$ is the set of couplings in $\mathfrak{P}([0,1]^4)$ that can be identified with  tensors $\boldsymbol{\pi}$ of dimension $n \times m \times \tilde{n} \times \tilde{m}$ such that
\begin{equation}\label{eqq1}
 \sum_{k,l} \pi_{ijkl}  = p_{ij}, \quad \sum_{i,j} \pi_{ijkl}  = \tilde{p}_{kl}.
\end{equation}
Note that the objective in \eqref{eqq2} with constraints \eqref{eqq1} is a linear problem which can be solved numerically by linear programming. 

To investigate empirically the  distribution of $\mathcal{W}(\P,\hat{\P}_N)$ we fix a reference measure $\P=(\vect{x},\vect{p})$ with   support  being a regular Cartesian product of equidistant points $x_i,y_i = (i-1)/(n-1)$, $i=1,\dots,n$, that is, of size $n\times n$. Figure  \ref{f1} shows an example measure for $n=3$.
\begin{figure}[h!] 
	\centering
		\includegraphics[width=10cm]{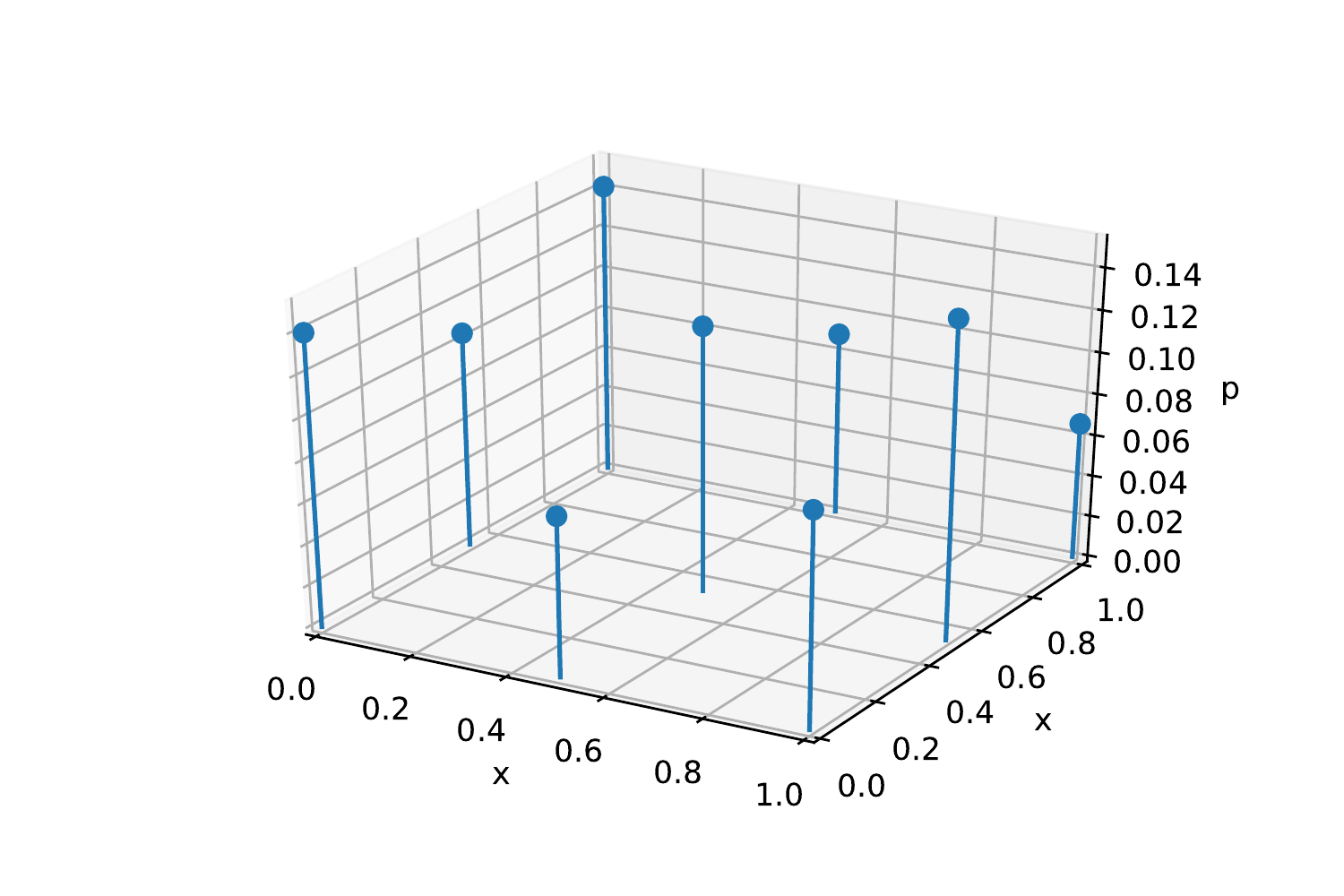}
	\caption{A probability measure on $[0,1]^2$ supported on the vertices of  a $3\times3$  grid.}
\label{f1}
\end{figure}

Next, we sample points $X_1,\dots,X_N$ from to $\P$ and associate the empirical measure $\hat{\P}_N=(\vect{x},\hat{\vect{p}}_N)$, where $\hat{\vect{p}}_N$ has elements
$$ \hat{p}_{ij}= \frac{ \#\{X_k:X_k=\vect{x}_{ij}\} }{N}.$$
Using this, we can compute a sample point of $\mathcal{W}(\P,\hat{\P}_N)=:\mathcal{W}(\vect{p},\hat{\vect{p}}_N)$.  
If we repeat this procedure $M$  times we obtain an $M$-sample of the distance from which we can compute an empirical distribution function $\hat{F}_{\mathcal{W}(\vect{p},\hat{\vect{p}}_N)}$. Figure \ref{f2} shows the result with the  measure  in Figure \ref{f1} as reference for $N=100$ and $M=500$ generated outcomes of the Wasserstein distance with $\ell_1$-norm. 
To compute the latter, we use the interior-point method of Andersen and Andersen \cite{andersen2000mosek} to solve the linear program.
\begin{figure}[h!] 
	\centering
		\includegraphics[width=8cm]{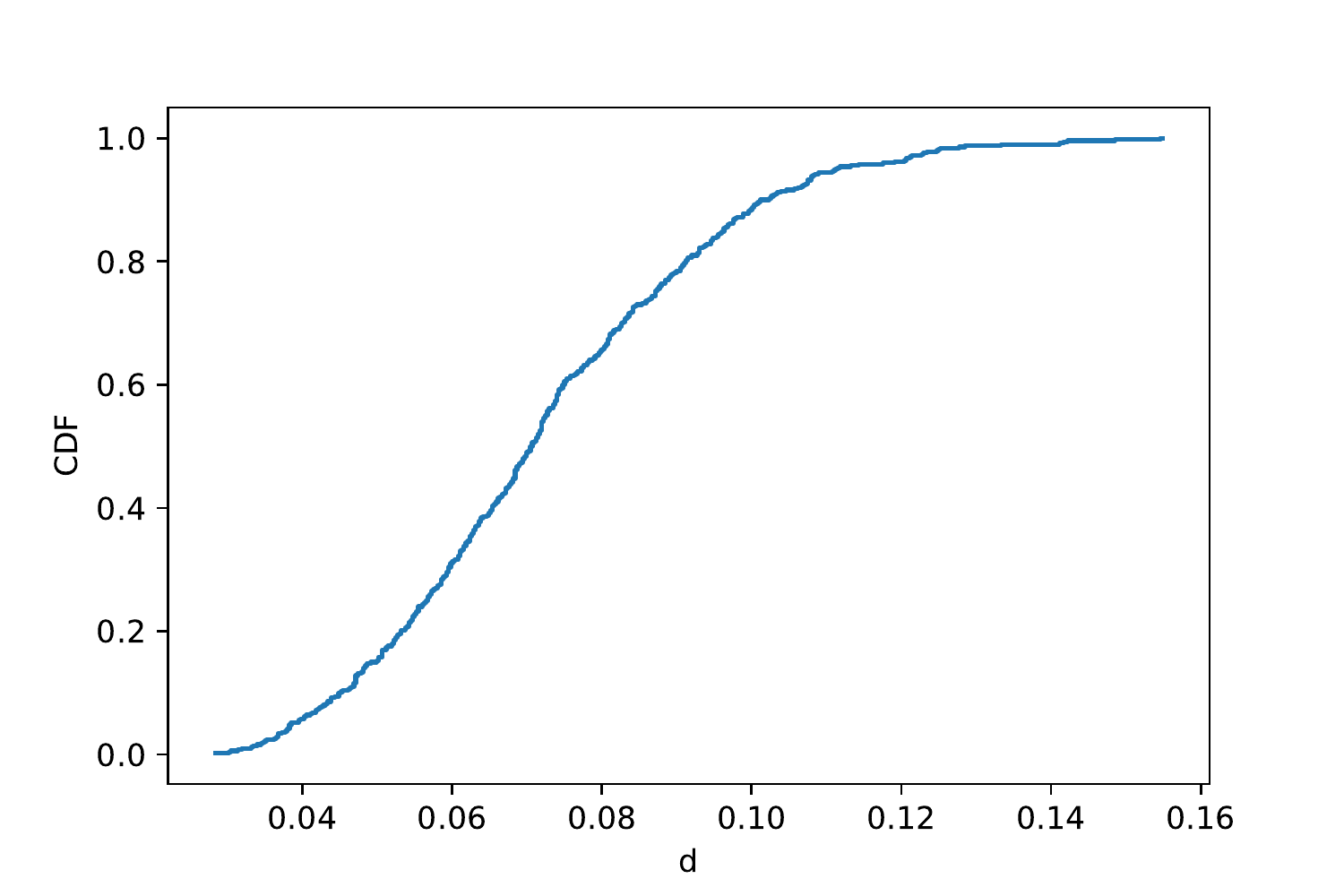}
	\caption{Empirical distribution function of $\mathcal{W}(\P,\hat{\P}_{100})$, from $M=500$  draws, with the measure shown in Figure \ref{f1} as generating measure $\P$.}
\label{f2}
\end{figure}

For a fixed confidence level $\alpha\in[0,1]$, we can now approach optimising the empirical quantile function
\begin{equation}\label{eqx1}
\text{argmax}_{\vect{p}\in \mathfrak{P}(\vect{x})} \hat{F}^{-1}_{\mathcal{W}(\vect{p},\hat{\vect{p}}_N)}(\alpha)
\end{equation}
with a numerical optimiser. 

Note, however, that the global problem \eqref{eqx1} is non-trivial. The objective is  high dimensional, expensive to evaluate  (each sample-point of the distance is  computed with a numerical linear-program solver) and with estimation error---in all making a case for noisy, gradient-free  optimisation. To this end, we employ Bayesian optimisation, see e.g. Osborne, Garnett and Roberts \cite{osborne2009gaussian} or Shahriari et al \cite{shahriari2016taking} for a review.  These methods use Bayesian theory to  frame the optimisation as a  sequential decision problem  through a probabilistic surrogate of the ``black-box'' objective function. The surrogate --- in our case a noisy Gaussian process model (see Williams and Rasmussen \cite{rasmussen2003gaussian}) of the empirical quantile function --- is used to carefully select a set of unseen arguments as candidates for an optimum, where the selection is with a Bayesian  expected-loss criterion. The true objective function is then evaluated at  selected locations, where a new potential optimum  is recorded. The surrogate model is also updated and re-estimated with acquired data; in our case the Gaussian process posterior and its hyper-parameters. The sequential procedure  continues until a  stopping criterion kicks in, here a given number  of function evaluations.
\begin{figure}[h!] 
	\centering
		\includegraphics[width=10cm]{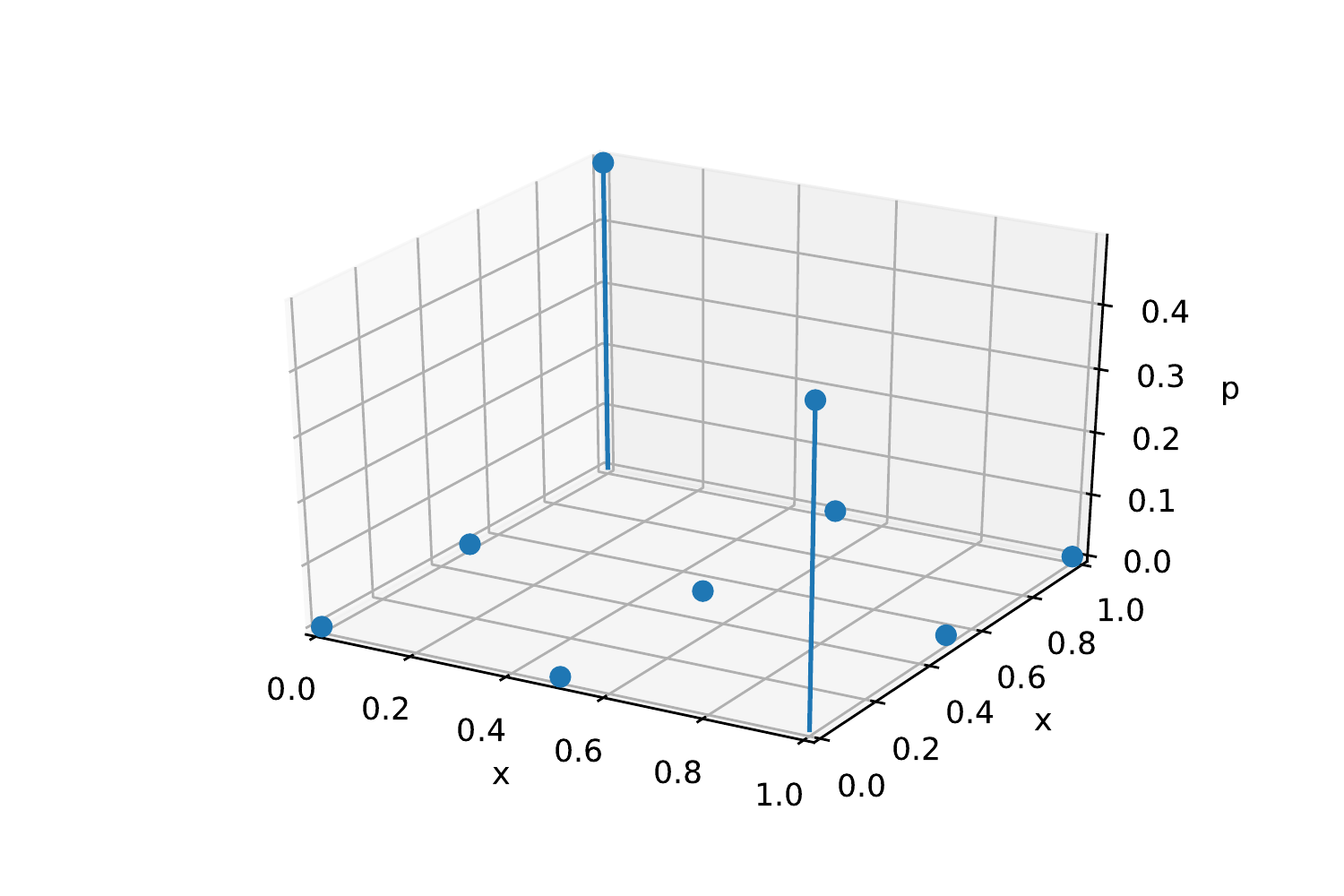}
	\caption{A maximising measure of the empirical quantile function of $\mathcal{W}(\P,\hat{\P}_{100})$ at level $\alpha=0.95$, from $M=100$ draws.}
\label{f3}
\end{figure}

Figure \ref{f3} shows a maximising measure of $\hat{F}^{-1}_{\mathcal{W}(\P,\hat{\P}_{100})}(0.95)$    where the distribution function is estimated  from $M=100$ draws of the distance. Note that the  quantile is maximised by an extremal measure with non-zero and equal mass at two points only, where the points are placed at two opposite corners. This is the case for high confidence levels: in Figure \ref{f7} we include the results for $\alpha \in \{0.75,0.80,0.85,0.95\}$. Note in particular  $\alpha=0.85$, where the non-zero point-masses are placed  at the other pair of opposite corners.
\begin{figure}[h!]
\centering
		\includegraphics[width=6cm,trim=80 30 30 0,clip]{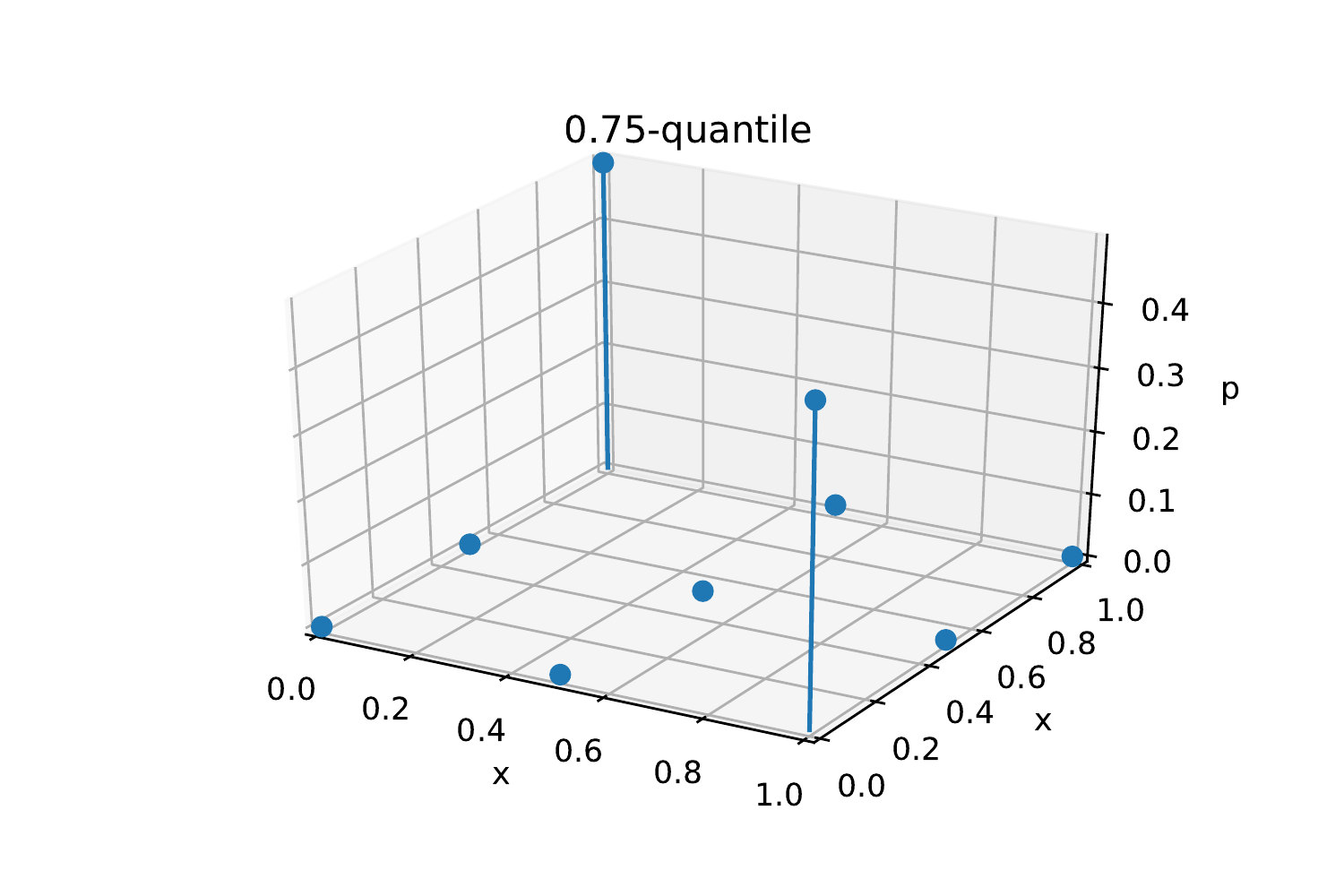} 
		\includegraphics[width=6cm,trim=80 30 30 0,clip]{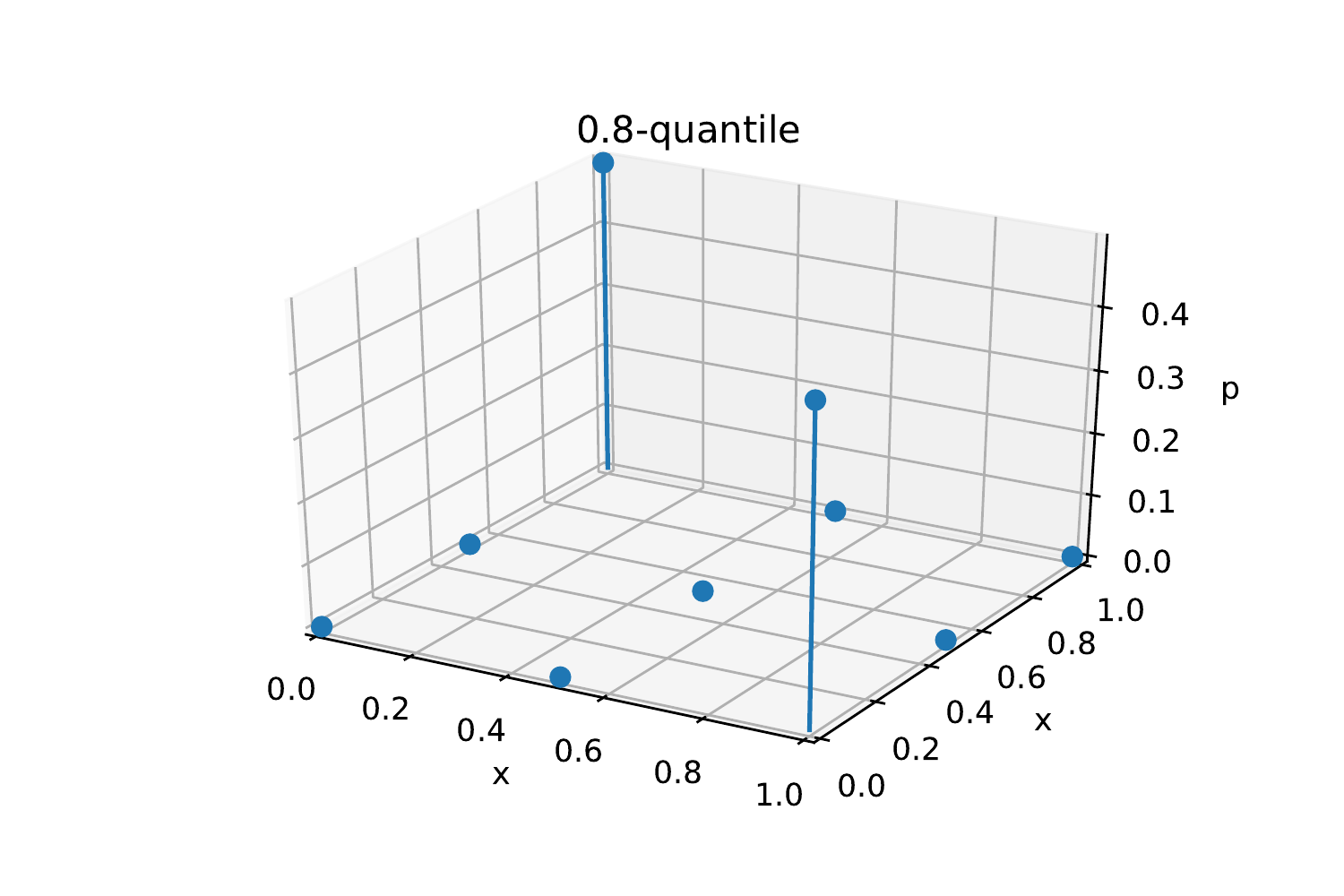} \\
		\includegraphics[width=6cm,trim=80 30 30 0,clip]{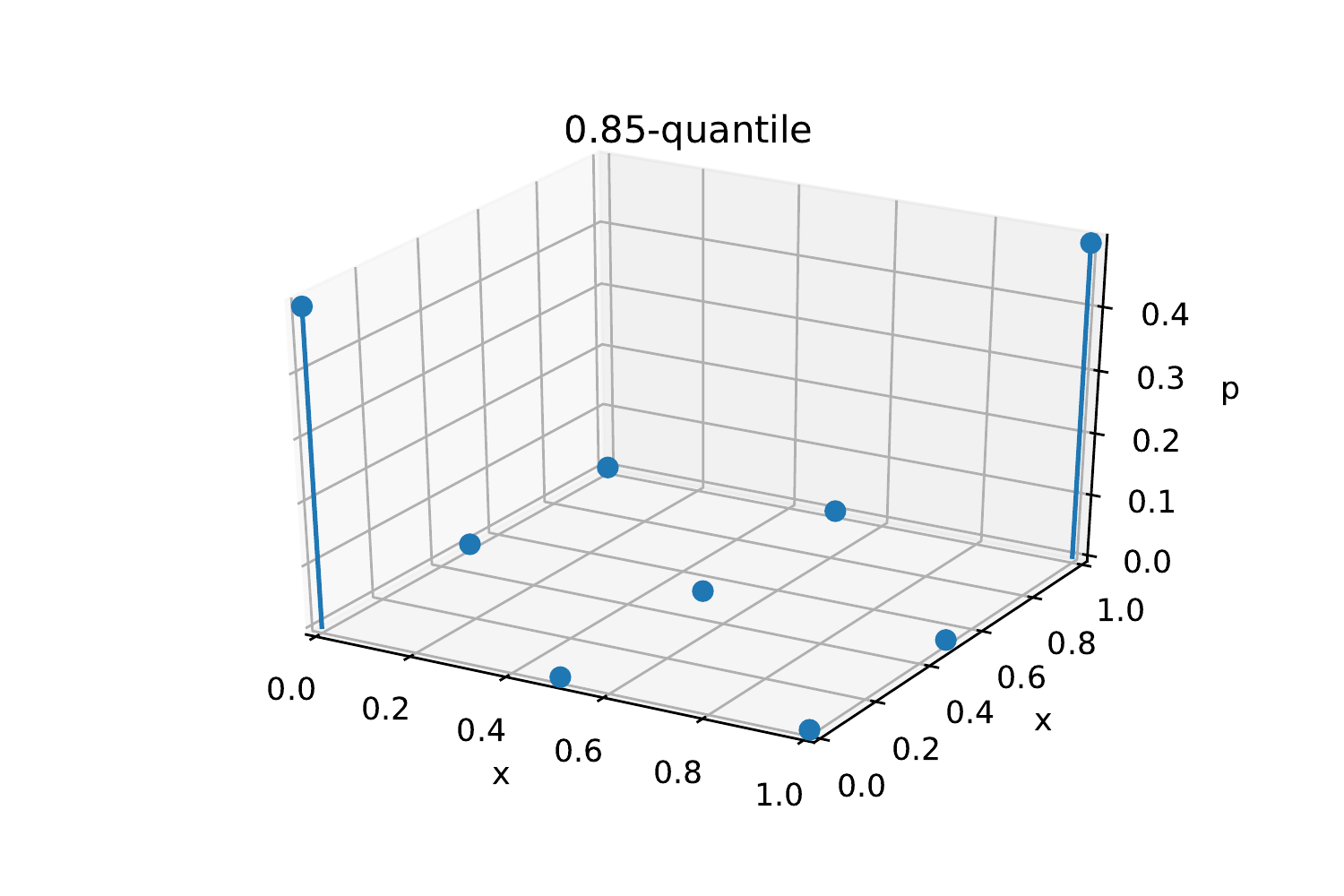} 
		\includegraphics[width=6cm,trim=80 30 30 0,clip]{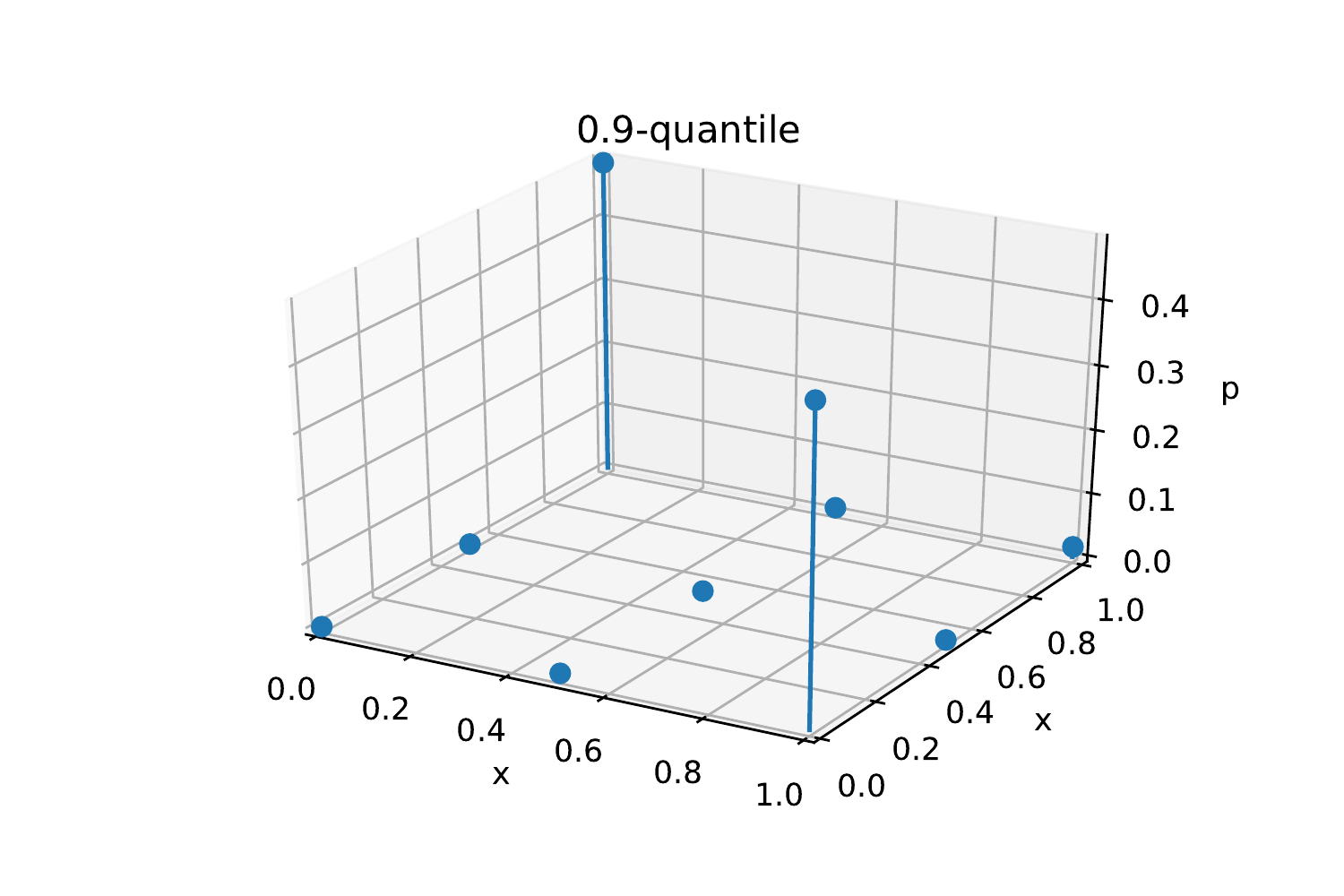}

\caption{Maximising measures of the empirical quantile function of $\mathcal{W}(\P,\hat{\P}_{100})$ at different confidence levels, from $M=100$ draws.}
\label{f7}
\end{figure}
Similarly, for low confidence levels $\alpha \in \{0.05,0.10\}$ the result is shown in Figure \ref{f8}. Here the estimated quantiles are maximised by measures with mass at all points of the support. Although not completely uniform, the apparent deviation from such a measure is likely due to estimation- and numerical errors: either in the empirical distribution, or in the  solution of the optimisation problem (or in both).
\begin{figure}[h!]
\centering
		\includegraphics[width=6cm,trim=80 30 30 30,clip]{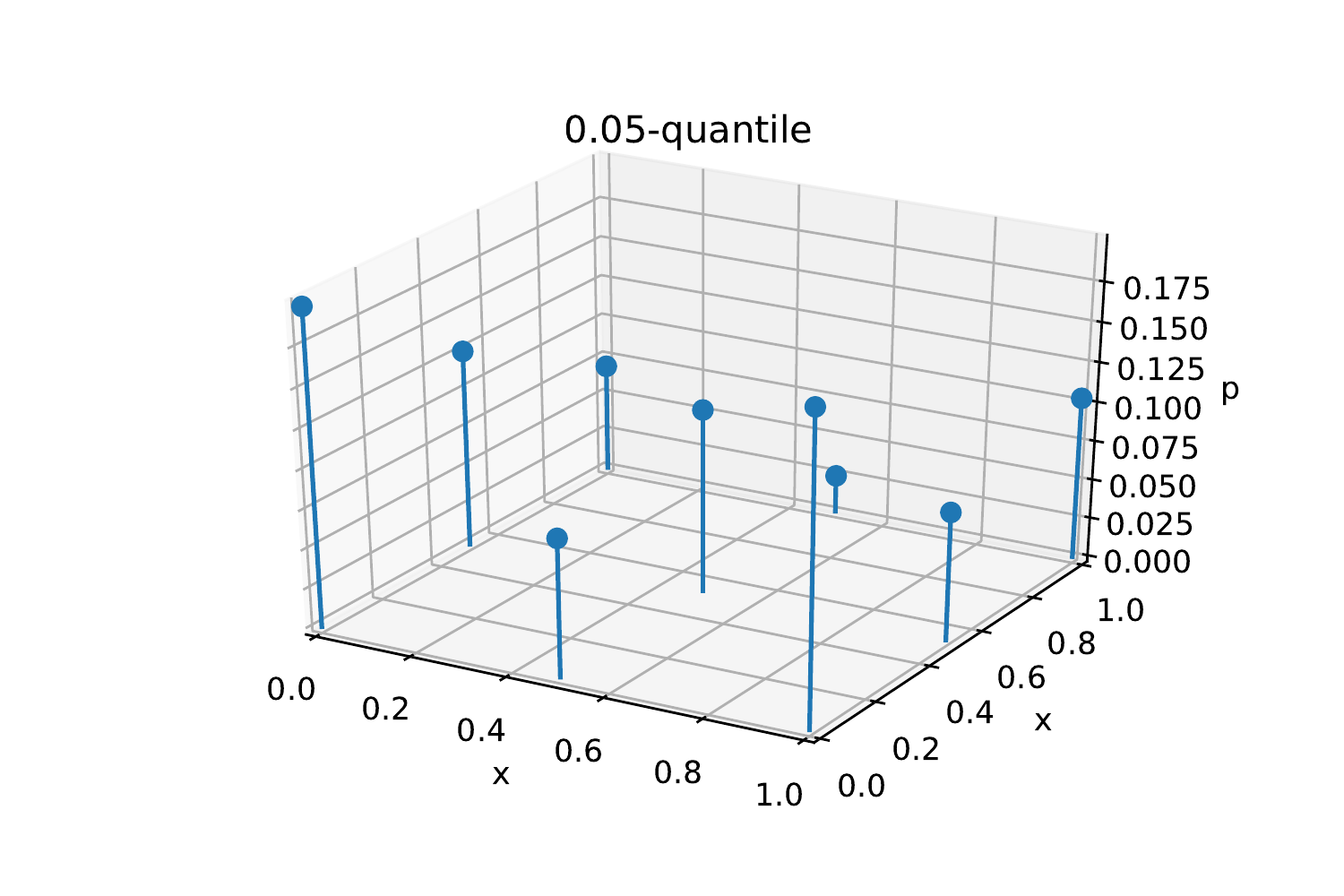}
		\includegraphics[width=6cm,trim=80 30 30 30,clip]{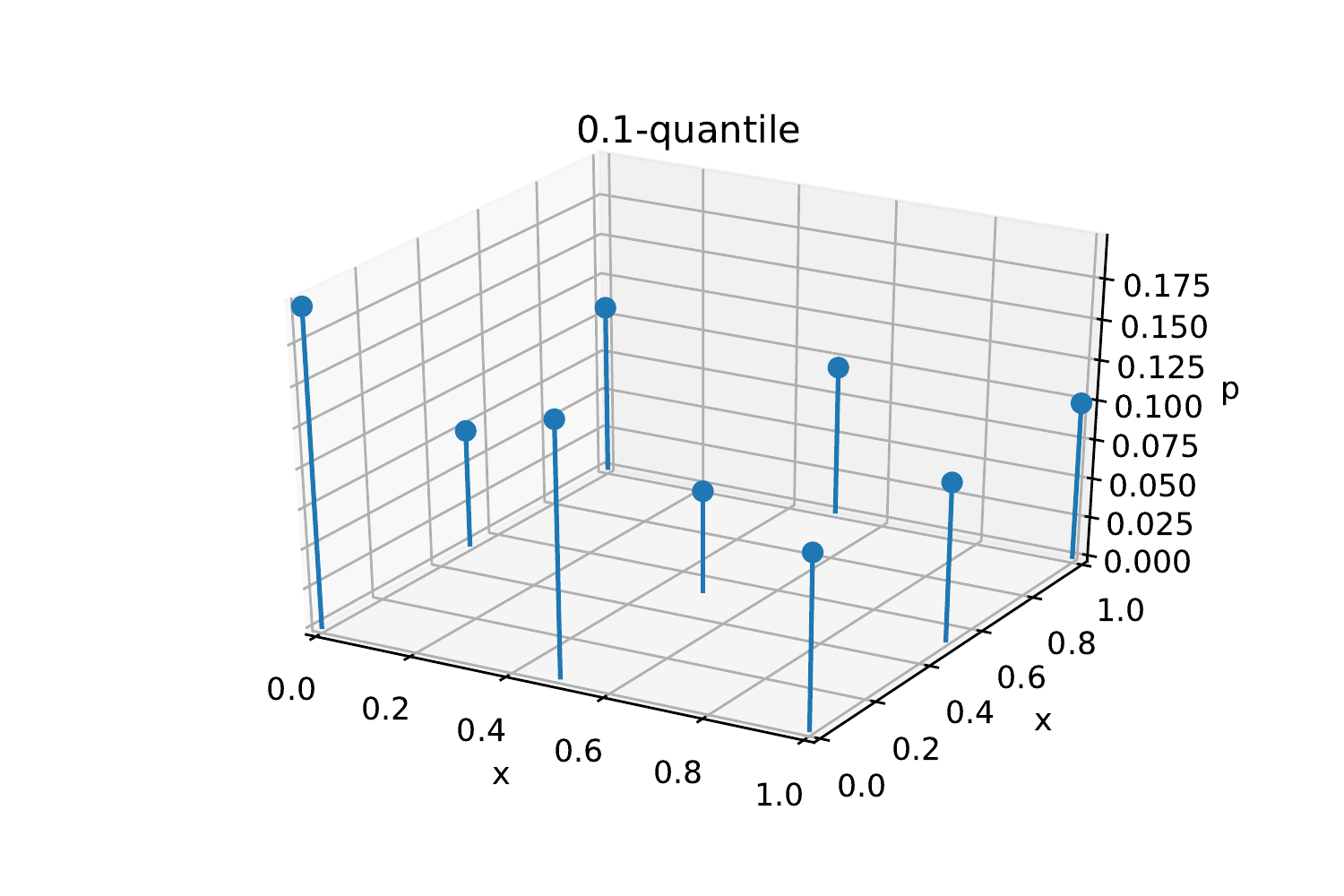} 
\caption{Maximising measures of the empirical quantile function of $\mathcal{W}(\P,\hat{\P}_{100})$ at different confidence levels, from $M=100$ draws.}
\label{f8}
\end{figure}

In any case, from the numerical results in $\R^2$ and comparison with Theorem \ref{Theorem::1} we derive the following conjecture:
\begin{conj}\label{conj:1}
Let $d\ge 1$, $K\subseteq \R^d$ be a convex compact set such that $0<\mu(K)$, where $\mu$ denotes the Lebesgue measure on $\R^d$. Furthermore let $x_0, x_1\in \R^d$ attain
\begin{align*}
\mathrm{argmax}\{\|y_0-y_1\|_{\ell_1} \ : \ y_0,y_1\in K\}
\end{align*}
Let $\alpha\in[0,1]$ and let $\ \mathcal{U}(K)$ denote the uniform distribution on $K$. Take $\epsilon>0$. Then there exists $\lambda=\lambda(\alpha, \epsilon)\in[0,1]$ such that 
\begin{align}\label{eq:main2}
\P^{\lambda}\coloneqq\lambda \frac{\delta_{x_0} +\delta_{x_1}}{2}+(1-\lambda)\ \mathcal{U}(K)
\end{align}
asymptotically maximises the normalised quantile $N^{1/2}F^{-1}_{\mathcal{W}(\hat{\P}_N,\P)}(\alpha)$, i.e. there exists $N_0\in \N$ such that for all $N\ge N_0$  $$\sup_{\P\in\mathfrak{P}(K)}F^{-1}_{\mathcal{W}(\hat{\P}_N,\P)}(\alpha)-F^{-1}_{\mathcal{W}(\hat{\P}^{\lambda}_N,\P^{\lambda})}(\alpha)\le \frac{\epsilon}{N^{1/2}}.$$
Furthermore $\lim_{\alpha\downarrow 0}\lambda=0$ and $\lim_{\alpha\uparrow 1}\lambda=1$. 
\end{conj}

A proof of such a multidimensional extension of Theorem \ref{Theorem::1} in the spirit of Conjecture \ref{conj:1} does not seem straightforward to us, at least given the tools we use: While Section \ref{sec:proof} is mainly concerned with controlling the quantisation error and should have a multidimensional analogue, we rely on a representation of $\mathcal{W}(\hat{\P}_N,\P)$ through the integral over cumulative distribution functions $F_{\hat{\P}_N}$ and $F_{\P}$ in the course of the whole paper. This goes back to the theory of optimal transport, and seems to be a fundamentally one-dimensional result.\\ For the same reasons our method of proof only covers the case of the $1$-Wasserstein distance, while we expect a similar result to hold for the $p$-Wasserstein distance for $p>1$ also. We refer to Johnson \& Samworth   \cite[Cor. 2.7(b), p.8]{samworth2004convergence} for a convergence result corresponding to Lemma \ref{Lemma::gine} for $p>1$ for finitely supported measures.

  \section{Applications \& outlook}\label{sec:applications}
For $\P=(\delta_0+\delta_1)/2$ clearly  $\int_0^1 |B(F_{\P}(t))|dt =|B(1/2)|$ and in particular
\begin{align*}
\P^\otimes(|B(1/2)|\ge t)&=2(1-\P^\otimes(B(1/2)\le t))=2-2\Phi(2t),
\end{align*}
where $\Phi$ denotes the cumulative distribution function of a standard normal random variable. Thus, intuitively, Theorem \ref{Theorem::1} shows that for \emph{any} measure $\P$, we have the approximate relation
\begin{align*}
(\P)^\otimes(\mathcal{W}(\hat{\P}_N, \P)\ge t)&=\P^\otimes\left(N^{1/2}\mathcal{W}(\hat{\P}_N, \P)\ge N^{1/2}t\right)\\
&\lesssim(\P^1)^\otimes\left(|B(1/2)| \ge N^{1/2}t\right)
 =  2-2\Phi\left(2\sqrt{N}t\right)
\end{align*}
for large $t\ge 0$ and large $N\in \mathbb{N}$.

On the other hand, for small $t\ge 0$ we are interested in $\P=\mathcal{U}([0,1])$ and so we need to compute the distribution of  $\int_0^1 |B(F_{\P}(t))|dt=\int_0^1 |B(t)|dt$. We refer to  Tolmatz \cite{tolmatz2000asymptotics,tolmatz2002distribution}, where the distribution of this integral has been studied.

These observations allow us to compute an explicit non-parametric confidence region for a sampling distribution $\P$, at least in the case $\alpha \approx 1$.

\begin{Theorem}\label{Theorem::Conf}
Let $X_1, \dots, X_N$ be i.i.d. samples from an unknown measure $\P \in \mathfrak{P}([0,1])$, with empirical distribution $\hat{\P}_N(\omega)$. For confidence levels $\alpha\to 1$, the random set of measures
\[\hat C_N(\omega;\alpha) = \left\{\P'\in \mathfrak{P}([0,1]): \mathcal{W}(\hat \P_N(\omega), \P') \le kN^{-1/2}\right\}\]
where $$k = \frac{\Phi^{-1}\left(\frac{1+\alpha}{2}\right)}{2}$$ is an asymptotic confidence region for $\P$; in the sense that
\[\liminf_{N\to \infty}\P^\otimes\Big(\P\in \hat C_N(\omega;\alpha)\Big)\geq 1-g(\alpha),\]
where $g:[0,1]\to [0,1]$ is a continuous function with  $g(\alpha) \sim 1-\alpha$ as $\alpha\to 1$.
\end{Theorem}
\begin{proof}
Let us first fix $n \in \N$ and let us recall the finitely supported measures $\P^{\bar{\lambda}(t,n)}$ from the proof of Theorem \ref{Theorem::1}.
From Remark \ref{rem:quantileapprox}, there is a continuous function $\bar C\ge 1$ with $\bar C(t)\to 1$ as $t\to \infty$ such that, for all $t\ge 0$ and all $\alpha\in[0,1]$,
\begin{align}
\label{eq:sam 1}
\begin{split}&\lim_{N\to \infty}(\P^{\bar{\lambda}(t,n)})^\otimes\big(\mathcal{W}(\hat{\P}_N(\omega), \P^{\bar{\lambda}(t,n)})\ge tN^{-1/2}\big)\\
 &\leq\lim_{N\to \infty} \bar C(t)(\P^1)^\otimes\big(\mathcal{W}(\hat{\P}_N(\omega), \P^1)\ge tN^{-1/2}\big).
\end{split}
\end{align}
As described above, we also know the distribution of $\mathcal{W}(\hat \P_N(\omega), \P^1)$ under $(\P^1)^\otimes$, that is,
\begin{align}\label{eq:sam 2}
\lim_{N\to \infty}(\P^1)^\otimes\left(\mathcal{W}(\hat{\P}_N(\omega), \P^1)\ge tN^{-1/2}\right)  =  2-2\Phi(2t).
\end{align}
Define
\[g(\alpha) := \bar C\left(\frac{\Phi^{-1}\left(\frac{1+\alpha}{2}\right)}{2}\right)(1-\alpha)=\bar C(k)~(2-2\Phi(2k)),\]
where $k$ is as in the statement of the theorem. Observe that $g$ satisfies $g(\alpha) \sim 1-\alpha$ as $\alpha \to 1$.
For large $N$,  we know from \eqref{eq:sam 1} and \eqref{eq:sam 2} that $k$ satisfies
 \[\lim_{N\to \infty}(\P^{\bar{\lambda}(k,n)})^\otimes\left(\mathcal{W}\left(\hat\P_N(\omega), \P^{\bar{\lambda}(k,n)}\right)\ge kN^{-1/2}\right)\leq g(\alpha) .\]
 From the approximation results in Section \ref{sec:proof}, for any $\P$, it follows that 
 \begin{align*}
& \limsup_{N\to \infty}\P^\otimes\Big(\P\not\in \hat C_N(\omega, \alpha)\Big)\\
&= \limsup_{N\to \infty}\P^\otimes\Big(\mathcal{W}(\hat{\P}_N(\omega), \P)\ge kN^{-1/2}\Big)\\
 &\le \limsup_{n \to \infty}\limsup_{N\to \infty}(\P^{\bar{\lambda}(k,n)})^\otimes\left(\mathcal{W}\left(\hat{\P}_N(\omega), \P^{\bar{\lambda}(k,n)}\right)\ge kN^{-1/2}\right)\\
& \leq g(\alpha),
 \end{align*}
  as desired.

%
%
\end{proof}

It is worth considering how these intervals compare with classical parametric estimation. 
\begin{Example}
 Let $X$ and $X_1,..., X_N$ be iid random variables from some distribution on $[0,1]$. For a function $f\in \mathrm{Lip}_1$ (that is $|f(x)-f(y)|\leq |x-y|$), define $Y= f(X)$ (By scaling, similar results hold for $K$-Lipschitz functions for any $K>0$).  Observe that $\mathrm{Var}(Y)\leq 1/4$ (from Popoviciu's inequality for variances), and that this bound is attainable for appropriate choices\footnote{In particular, Popoviciu's inequality is attained when $X$ has distribution of the form $(\delta_0 + \delta_{1})/2$ and $f(x) = x+y$ for some $y\in \R$. This is the same distribution achieving our extreme quantiles in Theorem \ref{Theorem::1} for large $\alpha$.} of $\P$ and $f$.
 
A classical confidence bound for $\E[Y]$ (using this worst-case bound on the variance and the central limit theorem) given observations $X_1,...,X_N$ is 
\[\frac{1}{N}\sum_{i=1}^N f(X_i) - \frac{\Phi^{-1}((1+\alpha)/2)}{2\sqrt{N}}\leq \E_\P[Y]\leq \frac{1}{N}\sum_{i=1}^N f(X_i) + \frac{\Phi^{-1}((1+\alpha)/2)}{2\sqrt{N}}.\]

Using an approach based on the Wasserstein distance $\mathcal{W}$, the Kantorovich--Rubinstein duality formula states that, for any $\P'\in \mathfrak{P}([0,1])$,
\[\mathcal{W}(\P', \hat\P_N) = \sup_{f \in \mathrm{Lip_1}}\Big\{\big|\E_{\P'}[f(X)] - \E_{\hat\P_N}[f(X)]\big|\Big\}.\]
So, for our specific choice of $f$, we obtain the corresponding bounds
\[\E_{\hat\P_N}[f(X)] - \mathcal{W}({\P'}, \hat\P_N)\leq \E_{\P'}[Y] \leq  \E_{\hat\P_N}[f(X)] + \mathcal{W}({\P'}, \hat\P_N).\]
Theorem \ref{Theorem::Conf} states that we have approximately $\alpha$-confidence that the true sampling distribution $\P$ lies within the region $\hat C = \{{\P'}:\mathcal{W}({\P'}, \hat\P_N)\le k/\sqrt{N}\}$. Substituting in the value of $k$ and evaluating $\E_{\hat\P_N}[f(X)]$, we obtain the confidence bounds
\begin{equation}\label{eq:confidence}\frac{1}{N}\sum_{i=1}^N f(X_i)  - \frac{\Phi^{-1}((1+\alpha)/2)}{2\sqrt{N}}\leq \E_{\P}[Y] \leq  \frac{1}{N}\sum_{i=1}^N f(X_i)  + \frac{\Phi^{-1}((1+\alpha)/2)}{2\sqrt{N}}.
\end{equation}

These non-parametric bounds are identical to those obtained through classical parametric approximations for every choice of $f$. On the one hand, by deriving them using a non-parametric approach, we obtain a bound which is uniform in  $f\in \mathrm{Lip}_1$, that is, if $\hat C_f$ denotes the interval given by \eqref{eq:confidence},  we have 
\[\P^\otimes\Big( \E_{\P}[f(X)]\in \hat C_f(\omega) \text{ for all }f\in \mathrm{Lip}_1\Big) \geq \alpha-\epsilon_\alpha,\]
whereas the classical approach only gives us an interval for a specific choice of $f$.  It is somewhat surprising that the interval computed in this non-parametric manner, for a bound which is uniform in $f$, is no wider than that obtained earlier for a specific $f$. In particular, no Bonferroni-type correction of the interval width is required (cf. for example, the classical methods of Dunn \cite{dunn61}).

On the other hand, the Wasserstein distance derived bound is only asymptotic for large $\alpha$ (although our numerical results suggest this approximation is very good for $\alpha>0.7$) and large $N$ (this is also true for those described in the parametric case, as they are based on the central limit theorem). Tighter bounds can be obtained in the parametric case by estimating the variance of $Y$ from observations.
\end{Example}


\section*{Acknowledgements}
Samuel Cohen and Martin Tegn\'{e}r thank the Oxford--Man Institute for research support. Samuel Cohen acknowledges the support of The Alan Turing Institute under the Engineering and Physical Sciences Research Council grant EP/N510129/1. Johannes Wiesel and Samuel Cohen gratefully acknowledge support from the European Research Council under the European Union's Seventh Framework Programme (FP7/2007-2013) / ERC grant agreement no. 335421. Johannes Wiesel furthermore acknowledges support of the German Academic Scholarship foundation and of the Erwin Schr\"odinger Institute for Mathematics and Physics.
\bibliographystyle{plain}
\bibliography{bib}
\end{document}